\documentclass[3p,11pt,authoryear]{elsarticle}

\usepackage{amssymb}
\usepackage[dvipsnames]{xcolor}  
\usepackage{soul}
\usepackage{multirow}
\usepackage{mathtools} 
\let\oldforall\forall
\let\forall\undefined
\DeclareMathOperator{\forall}{\oldforall}
\usepackage{amsmath}
\usepackage{blindtext}
\usepackage{nccmath}
\usepackage{amsthm}
\usepackage{multicol}
\usepackage{subcaption}
\usepackage{float}
\usepackage{lscape}
\usepackage{xcolor,soul}
\usepackage[normalem]{ulem}
\usepackage{changepage}
\usepackage{algorithm}
\usepackage{comment}
\usepackage{float}
\usepackage{setspace}
\usepackage[figuresright]{rotating}
\usepackage{graphicx}

\usepackage{lineno}
\usepackage{natbib}
\usepackage{makecell}
\usepackage{booktabs}
\usepackage{array}
\usepackage{bbm}

\DeclareUnicodeCharacter{00A0}{~}
\DeclareMathOperator*{\argmin}{argmin}
\usepackage{algorithm}
\usepackage[noend]{algpseudocode}
\usepackage[english]{babel}

\newtheorem{prop}{Proposition}

\theoremstyle{remark}

\DeclareSymbolFont{matha}{OML}{txmi}{m}{it}
\DeclareMathSymbol{\varv}{\mathord}{matha}{118}

\newcommand{\Mod}[1]{\, (\text{mod}\ #1)}
\makeatletter
\def\BState{\State\hskip-\ALG@thistlm}
\makeatother
 
\usepackage[hidelinks]{hyperref}

\journal{Transportation Science}

\begin{document}

\begin{frontmatter}





\title{Average Distance of 
Random Bipartite Matching in One-dimensional Space and Networks
}


\author[label1]{Yuhui Zhai}
\author[label1]{Shiyu Shen}
\author[label1]{Yanfeng Ouyang}

\address[label1]{Department of Civil and Environmental Engineering, University of Illinois at Urbana-Champaign, Urbana, IL 61801, USA}

\begin{abstract}
The bipartite matching problem is widely applied in the field of transportation; e.g., to find optimal matches between supply and demand over time and space. Recent efforts have been made 
on developing analytical formulas to estimate the expected matching distance in bipartite matching with randomly distributed vertices
in two- or higher-dimensional spaces, but no accurate formulas currently exist for one-dimensional problems. 
This paper presents a set of closed-form formulas, without curve-fitting, that can provide accurate average distance estimates for one-dimensional random bipartite matching problems (RBMP). We first focus on a lattice case and propose a new method that relates the corresponding matching distance to the area size between a random walk path and the x-axis. 
This result directly leads to a straightforward closed-form formula for balanced RBMPs. 
For unbalanced RBMPs on a lattice, we first analyze the properties of an unbalanced random walk that can be related to balanced RBPMs after optimally removing a subset of unmatched points, and then derive a set of approximate formulas. 
Additionally, we build upon an optimal point removal strategy to 
derive a set of recursive formulas that can provide more accurate estimates.
Then, we extend the results to three problem variants, including RBMPs with periodic boundaries, uniformly distributed points, and arbitrary-length line.
Last, we shift our focus to regular networks, and use the one-dimensional results as building blocks to derive RBMP formulas.
To verify the accuracy of the proposed formulas, a set of Monte-Carlo simulations are generated for a variety of matching problems settings. 
Results indicate that our proposed formulas provide quite accurate distance estimations for one-dimensional line segments and networks under a variety of conditions. 
\end{abstract}


\end{frontmatter}

\section{Introduction}
\label{sec: introduction}

The bipartite matching problem, defined on a bipartite graph with two disjoint set of vertices, is widely used in the field of transportation; e.g., to find optimal matches between supply and demand points over time and space in transportation services.
\textcolor{RoyalBlue}{
}
For example, in two-dimensional spaces, it can be used to model how surface vehicles (e.g., taxi) are matched to their customers, such as in ride-hailing or food delivery systems \citep{shen_zhai_ouyang_2024, Tafreshian_ridesharing_2020}. 
In three-dimensional spaces, it can be used to dynamically dispatch and reposition aerial vehicles (drones) for delivering goods in the air
\citep{Aloqaily_2022_UAV}, or to optimize the flying trajectories of drone swarms during take-off and landing \citep{Hernandez_2021_UAV}. 
In addition, bipartite matches that span the time dimension can be useful for finding the optimal schedules among a series of tasks \citep{Ding_FluidModel_2021, Af_queue_2022}, such as optimizing container transshipment among freight trains \citep{Fedtke_2017_rail}, or minimizing customer/vehicle waiting in reservation-based ride-sharing services \citep{shen_swap_2023}. 

Strictly speaking, bipartite matching applications in the transportation field are likely to be associated with a sparse network, where supply and demand points are distributed along network edges (after ignoring the local access legs) \citep{Abeywickrama_network_2022}. A special case would be one on one-dimensional transportation routes or corridors. 
For instance, 
it can be used to model the operations of drones that are used to deliver goods to customers located on different floors of a tall building \citep{Ezaki_2024_elevator, Seth_2023_highrise}, or vehicle-customer matching for a ride-hailing system on a single city corridor \citep{Nitish_resouceAllocation_2020}. 
In a more general setting, the matched points may be on different edges of a network, 
where the network topology can have a direct impact. 
Hence, it is particularly interesting to investigate how some of the key network features, such as edge density and node degree \citep{Boeing_2020_network, Wang_2020_rideshairngNetwork, Duan_2014_citynetwork}, may affect the matching results.


The bipartite matching problem in a network also has applications across many other science and engineering fields. In biology, it can be used to measure and analyze the similarity between heterogeneous genes
\citep{Zhang_bicliques_2014}, or to examine the structural relationship between proteins \citep{Wang_protein_2004}. 
In robotics, it can be applied to efficiently allocate tasks among multiple decentralized robots 
\citep{Ghassemi_robotMultiTask_2018}, or to design coordination mechanisms to prevent collisions along paths of multiple agents \citep{Dutta_robotPath_2017}. 
In computer science, it can be used to find optimal allocation of virtual machines to users, so as to reduce delay in mobile cloud computing \citep{Jin_schedule_2022}.

Any bipartite matching problem instance can be solved very efficiently using a range of well-known algorithms, including combinatorial optimization algorithms such as Hungarian algorithm \citep{Kuhn_Hungarian_1955} and Jonker-Volgenant algorithm \citep{Jonker_alg_1987}, or newly developed machine-learning based algorithms \citep{georgiev_neural_2020}. 
However, in the context of service and resource planning, one is often interested in estimating the average matching cost across various problem realizations, so as to evaluate the service efficiency under and resource investments. 
For example, in designing mobility services, the average matching distance, often referred to as the ``deadheading" distance, is a key indicator that captures the unproductive cost spent by both service vehicles (i.e., running without passengers) and customers (i.e., waiting for pickup). 
Analytical formulas that reveal the relationship between the average matching distance and vehicle/customer distribution, such as those developed in \cite{daganzo_1978} and \cite{yang_CobbDoglous_2010}, are often preferred by the service operators, because these formulas can not only provide valuable managerial insights, but also be directly incorporated into mathematical models to optimize service offerings. 
Similar formulas have been used to design system-wide operational standards (e.g., demand pooling time interval) for customer-vehicle matching \citep{shen_zhai_ouyang_2024}, evaluate the effectiveness of newly proposed customer matching strategies \citep{ ouyang_2023_swap, shen_swap_2023, Stiglic_2015_meeting}, or analyze the impacts of pricing and market competition on the social welfare \citep{wang_2016_pricing, zhou_2022_competition}. 

The need to estimate the expected bipartite matching distance for planning decisions has led to the exploration of a stochastic version of the problem, known as ``Random Bipartite Matching Problem (RBMP)" in the literature. 
Earlier studies in the field of statistical physics were among the first to explore such a problem. \cite{Mezard_replica_1985} used a ``replica method" to derive asymptotic formulas for the average optimal cost of a matching problem where the numbers of points in both subsets are nearly equal (i.e., balanced), and the edge weights identically and independently follow a uniform distribution.  
Building upon this work, \cite{Caracciolo_2014_scaling} and \cite{Caracciolo_2015_scaling} proposed scaling hypothesis and developed asymptotic matching distance approximations 
for balanced RBMPs in three- or higher-dimensional Euclidean spaces. 
However, these asymptotic approximations were derived under the strong assumption that the number of bipartite vertices approaches infinity, and hence could only serve as bounds rather than exact estimates when the number of vertices is small. More importantly, their proposed formulas require curve fitting that estimates coefficients from simulated data. 
\cite{daganzo_tlp_2004} studied a related problem, which they call the Transportation Linear Programming (TLP), and proposed approximated formula for estimating the average item-distance among points with normally distributed demands and supplies. Through probabilistic and dimensional analysis, they introduced a bound to estimate the solution in two- and higher dimensions. 
Very recently, \citet{shen_zhai_ouyang_2024} proposed a set of closed-form formulas for 
arbitrary numbers of points in both subsets, an arbitrary number of spatial dimensions, and arbitrary Lebesgue distance metrics. 
Their model has shown to provide very accurate estimates, without curve fitting, in spaces with higher than two dimensions. 
However, their approach ignores the boundary of the space as well as the resulted correlation among the matched pairs, which is reasonable for higher dimension spaces but causes notable errors in one-dimensional space especially 
when the the two subsets are (nearly) of equal size. 

For one-dimensional problems (i.e., along a line, such as the x-axis), the literature has extensively studied the balanced case when the edge cost between two vertices equals their distance raised to a power $p \in (-\infty,0) \cup (1, \infty)$. 
Among these, problems with $p > 1$ 
have received the most attention, as the following desirable property of the optimal matching solution holds: after sorting both sets of points respectively along the line, the $i$-th point from one set must be matched with the $i$-th point from the other set \citep{Caracciolo_2014_matching}. 
Building on this property, \citet{Caracciolo_2014_matching} derived asymptotic formulas in the joint limit as both the number of points and $p$ go to infinity;
\citet{Caracciolo_2017} extended the results to a larger class of strictly increasing convex cost functions and those 
ensuring monotonicity of shifted difference operators; and \citet{Caracciolo_2019} investigated the case with a finite number of points and $p > 1$. 
However, for $p \in (0, 1]$, including $p=1$ (i.e., 
the edge cost is equal to the exact distance), the problem remains an open challenge. 
In a separate line of inquiry, 
\cite{daganzo_tlp_2004} proposed an exact formula for the average minimal item-distance among points with normally distributed demand and supply in a one-dimensional space. Their results cannot be directly applied to RBMP, because they assume that (i) the supply and demand points are balanced, and that (ii) the supply or demand value associated with each randomly distributed point follows a normal distribution, while in RBMP, these values are either positive one or negative one. 
Nevertheless, their analysis provides very strong insights. One of their key findings is that the total minimal item-distance for all points is equal to the size of the area enclosed by the cumulative supply curve and the x-axis. This essential property also holds for our one-dimensional RBMP.

To the best of our knowledge, no existing formula can provide accurate estimates for RBMPs (with arbitrary numbers of bipartite vertices) in a one-dimensional space or in a network. 
This paper aims to to fill this gap by introducing a set of closed-form formulas (without curve-fitting) that can provide sufficiently accurate estimates. 
This is done in three steps. First, we estimate the expected optimal matching distance for a special one-dimensional RBMPs, where points can only be located on a regular lattice inside a unit-length line interval along the x-axis. 
Our proposed method relates the matching distance in a balanced RBMP to the enclosed area size between the path of a random walk and the x-axis, and derives a closed-form formula for the optimal matching distance for balanced RBMPs. 
For the more challenging unbalanced RBMPs, a closed-form approximate formula is developed by analyzing the properties of an unbalanced random walk and the optimal way to remove a subset of excessive points. 
Additionally, a feasible point removal and swapping process is proposed to develop a set of recursive formulas that are more accurate. 
Second, we study three variants of one-dimensional RBMP: one with periodic boundaries, one with uniform point distribution, and one in an arbitrary-length line interval.
The insights from these variants are used as building blocks in our third step to derive formulas for RBMPs on a regular network, where all points are generated from spatial Poisson processes along the edges.
The expected optimal matching distance is derived as the expectation across two probabilistic matching scenarios that a point may encounter: (i) the point is locally matched with a point on the same edge, or (ii) it is globally matched with a point on another edge. 
To verify the accuracy of proposed formulas, a set of Monte-Carlo simulations are conducted for a variety of matching problem settings, for both one-dimensional and network problems. 
The results indicate that our proposed formulas have very high accuracy in all experimented problem settings. 
The proposed distance estimates, in simple closed forms, could be directly used in mathematical programs for strategic performance evaluation and optimization.

The remainder of this paper is organized as follows. Section \ref{sec: model} below first presents models and formulas for 
balanced and unbalanced one-dimensional RBMPs (as a building block) 
when the points are distributed on a lattice.
Section \ref{sec: variant} presents formulas for three problem variants, including periodic boundary, uniform point distribution, and arbitrary-length line. 
Section \ref{sec: network model} then presents formulas for the problem in regular networks.
Section \ref{sec: numerical} then presents numerical experiments to validate the proposed formulas.
Finally, Section \ref{sec: Conclusion} provides concluding remarks and suggests future research directions.

\section{1D RBMP Distance on A Lattice}
\label{sec: model}

\subsection{Problem Definition}
\label{subsec: problem definition}
We begin by defining the unit-length one-dimensional RBMP 
on a lattice and open boundaries.
Two sets of points, with given respective cardinalities $n \in \mathbb{Z}^+$ and $m \in \mathbb{Z}^+$, are distributed inside a line segment defined over the interval $[0,1]$ along the x-axis.
Without loss of generality, we assume $n \geq m$. 
\textcolor{Black}{The line segment is evenly divided into a lattice of $n+m+1$ sub-intervals with equal length $l = \frac{1}{n+m+1}$. 
The two sets of points 
are randomly mixed, sorted, and sequentially assigned to the right endpoints of the first $n+m$ sub-intervals without any overlap.}
For each realization of the points' locations, denote $V$ and $U$ as the two point sets, where $|V| = n$ and $|U| = m$.
A bipartite graph can be constructed, whose set of edges $E$ connect every pair of points in the two sets; i.e., $E=\{(u,v): \forall u \in U, v \in V\}$. The weight on edge $(u,v)\in E$ is the absolute difference between the two points' coordinates $x_{u}$ and $x_v$, i.e., $\|x_{u}-x_v\|$. 
Since $n \geq m$, every point $u \in U$ can be matched with exactly one point $v \in V$. We let $y_{uv}=1$ if $u$ is matched to $v$, or $0$ otherwise. The objective is to find a set of matches $\{y_{uv}: \forall u\in U, v \in V\}$ that minimize the total matching distance, as follows: 
\begin{align} 
\label{eq: def}
\min &\sum_{u\in U, v\in V} y_{uv} \|x_{u}-x_v\|; \\
\label{eq: def_ctd}
\text{s.t. } & \sum_{v\in V} y_{uv} = 1, \forall u \in U; \quad 
\sum_{u\in U} y_{uv} \leq 1, \forall v \in V; \quad
y_{uv} \in \{0, 1\}, \forall u\in U, v\in V.
\end{align}

The average optimal matching distance across the realized points in $U$ is a random variable which depends on the random realization of $U$ and $V$. Its distribution is clearly governed by parameters $m$ and $n$, and hence we denote it $X_{m,n}$. 
We are looking for a closed-form formula for the expectation of the average optimal matching distance, $\mathbb{E}[X_{m,n}]$. In order to do that, we first show that $X_{m,n}$ 
can be estimated based on the enclosed area between the path of a related random walk and the x-axis, and then we take expectation of this enclosed area.
For a simple balanced problem (i.e. when $n = m$), we can directly derive a closed-form formula. 
For an unbalanced problem (when $n > m$), we first show optimality properties for the matching, and then build upon that to derive both a closed-form and a recursive formula.

\subsection{Random walk approximation}
\label{subsec: approximation}
For each realized instance of one-dimensional RBMP, let $I = U \cup V$. 
Sort all the points in $I$ by their x-coordinates 
between 0 and 1, and index them sequentially by $i$.
For point $i \in I$, denote $x_i \in [0,1]$ as its x-coordinate and $z_i$ as the value of its supply; i.e., $z_i=1$ if $i\in V$ (indicating a supply point), or $z_i = -1$ if $i\in U$ (indicating a demand point).
A cumulative ``net" supply curve can then be constructed for any coordinate $x$ and any subset of points $I' \subseteq I$; i.e., $S(x; I') = \sum_{\{\forall i\in I', x_i \leq x\}} z_i$. It is clearly a piece-wise step function. There are three special cases: $S(x; I)$ represents the net supply curve constructed by the full set of points in $I$; $S(0; I')$ and $S(1; I')$ represent the net supply values at both ends of the curve, $x=0$ and $x=1$, respectively. 

Every curve $S(x; I)$ can be related to the realized path of a specific type of one-dimensional random walk with $m+n$ steps, starting from $S(0; I) = 0$. 
Among these $m+n$ steps, exactly $n$ steps each increase the net supply by $1$ and $m$ each decrease the net supply by $1$; as such $S(1; I) = n-m$.
The locations of the points in $I$ correspond to the positions of these steps. 



Denote $A(x;I')= \sum_{\{\forall i\in I \setminus \{|I|\}, x_i \leq x\}} l\cdot |S(x_{i}; I')|$ 
as the total absolute area between curve $S(x; I')$ and the x-axis from 0 to $x$. 
Next we show how $\mathbb{E}[X_{m, n}]$ can be derived out of such an area, for both balanced and unbalanced matchings.




\subsection{Balanced case ($m=n$)}
We begin with the special case where $m=n$. Now set $I$ contains $2n$ points, and 
$A(1; I)$ is the total absolute area between curve $S(x;I)$ and the x-axis from 0 to 1. 
\cite{daganzo_tlp_2004} proved that $A(1; I)$ must equal the minimum total shipping distance of a one-dimensional TLP with equal number of supply and demand points, and thus it must also equal the minimum total matching distance of the corresponding one-dimensional RBMP instance. 
Thus, the expected optimal matching distance per point can be estimated by the following:
\begin{align} 
\label{eq: area_daganzo}
\mathbb{E}[X_{n, n}] = \frac{1}{n} \cdot \mathbb{E}[A(1; I)].
\end{align}

To further estimate $\mathbb{E}[A(1; I)]$, we first 
let $B(n)$ denote the expected area between the path of a standard random walk (i.e., whose step size equals unit length) and the x-axis. \cite{Harel_randomWalkArea_1993} has provided a formula for $B(n)$, as follows: 
\begin{align}
\label{eq: harel_unit_step}
B(n) = \frac{n2^{2n-1}}{\binom{2n}{n}}
\xrightarrow{n\gg 1} \frac{n\sqrt{\pi n}}{2}.
\end{align}
The last step of approximation, when $n \gg 1$, comes from Stirling's approximation. 
The basic intuition behind this formula is as follows. 
In a random walk with a fixed total number of steps (e.g., $2n$), there are a finite number of possible combinations of upward (e.g., $n+k$) steps and downward steps (e.g., $n-k$), when $k$ varies from 0 to $2n$. 
Next, for each $k$ value, the probability and expected area of a random walk can be determined: the probability is derived using backwards induction starting from some simple cases (e.g., $k=0$); the expected area, conditional on $k$, is computed by the absolute difference between the numbers of upward and downward steps, multiplied by the step size. 
(For example, for a random walk with $n+k$ upward steps and $n-k$ downward steps, its expected area between the curve and the x-axis can be show to be simply $2k$.) Then, Equation \eqref{eq: harel_unit_step} can be obtained by taking the unconditional expectation of these area sizes across all possible values of $k$.

Then, we can multiply $B(n)$ by \textcolor{Black}{the
step size $l=\frac{1}{2n+1}$ }to obtain
$\mathbb{E}[A(1; I)]$, and plug it into Equation \eqref{eq: area_daganzo} to obtain the following:
\begin{align}
\label{eq: exp_Xnn}
\mathbb{E}[X_{n, n}] 
= \frac{l \cdot B(n)}{n}
{\color{Black}{
= \frac{1}{2n+1} \cdot \frac{2^{2n-1}}{\binom{2n}{n}}}} 
\xrightarrow{n\gg 1} \frac{1}{4}\sqrt\frac{{\pi}}{n}.
\end{align}
\textcolor{Black}{Note here this formula has a constant time complexity $\mathcal{O}(1)$.}

\subsection{Unbalanced case ($n>m$)}

Next, we consider the 
unbalanced case where $n>m$, for which $n-m$ supply points from $V$ will remain unmatched for each realization. 
Let $V’ \subset V$ denote the set of these $n - m$ unmatched points. 
If we remove all points in $V'$ from $V$, the problem will reduce to a balanced one with an equal number (i.e., $m$) of demand and supply points. As a result, the values on the net supply curve after point removal, $S(x; I\setminus V')$, at $x=0$ and $x=1$ should both equal zero; i.e., 
\begin{align}
\label{eq: post_removal_S_1}
S(0; I\setminus V') = S(1; I\setminus V') = 0.
\end{align}
Figure \ref{fig: removal_process} shows an example of 
how such point removal process affects the net supply curve along the entire x-axis. 
In this figure, the points in $V$ and $U$ are represented by the red dots and blue triangles, respectively. The original and post-removal curves, $S(x; I)$ and $S(x; I \setminus V')$, are represented by the red dot-dash line and blue dashed line, respectively. The points in $V'$ and the net supply values at their corresponding coordinates on the original curve, $S(x_{v'}; I), \forall v'\in V'$, are marked by the black cross markers. 
Note that every time a point $v' \in V'$ is removed, the net supply values within $[x_{v'},1]$ will decrease by one. 
As a result, the cumulative reduction of net supply at position $x$ should be determined by the total number of removed points within $[0,x]$.
Sort the points in $V'$ by their x-coordinates, from left to right, as $\{v'_1, \dots, v'_{n-m}\}$. If we further denote $v'_0$ and $v'_{n-m+1}$ as two virtual points at the boundaries; i.e. $x_{v'_{0}}=0$ and $x_{v'_{n-m+1}}=1$, then the net supply values on the two curves within range $[x_{v'_k}, x_{v'_{k+1}})$ must satisfy the following relationship:
\begin{align}
\label{eq: post_removal_S_x}
S(x; I\setminus V') = S(x; I) - k, \forall k \in \{0, \cdots, n-m\}, x\in [x_{v'_k}, x_{v'_{k+1}}).
\end{align}
This relationship is illustrated in Figure \ref{fig: removal_process} by the black arrows. 

\begin{figure}[ht!] 
\centering 
\includegraphics[scale=0.7]{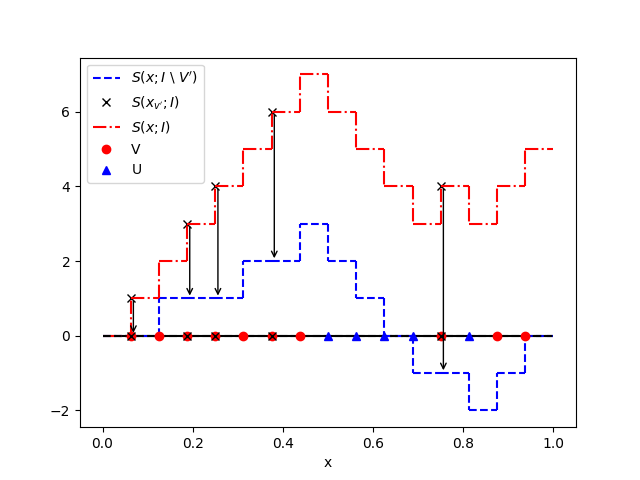}
\caption{Point removal process in an unbalanced problem.}
\label{fig: removal_process}
\end{figure}

Among all possible combinations of points in set $V'$, we denote $V^*$ as the optimal set of removed points that minimizes the area enclosed by the post-removal curve and the x-axis: 
\begin{align*}
V^{*}&= \argmin_{\forall V^{'}\subset V, \text{ } |V^{'}|=n-m } A(1; I\setminus V^{'}).
\end{align*}
The optimal post-removal area $A(1; I\setminus V^{*})$, enclosed by the optimal post-removal curve $S(1;I\setminus V^{*})$ and the x-axis, must equal the minimum total matching distance of the original unbalanced RBMP instance. 
Set random variable $Z_{m,n} = A(1; I\setminus V^{*})$ which depends on the random realization of $I= U\cup V$ (where $|V| = n$ and $|U| = m$), 
and then, 
similar to how we handle the balanced case: 
\begin{align}
\label{eq: exp_Xnm_incomplete}
\mathbb{E}[X_{m,n}] = \frac{1}{m} \cdot \mathbb{E}[Z_{m,n}].
\end{align}
In the following subsections, we will: (i) show the optimal post-removal curve must include a series of balanced random walk segments; 
(ii) derive an approximate closed-form formula for $\mathbb{E}[X_{m,n}]$ by estimating the area of each balanced segment; 
(iii) provide an alternative estimation for $\mathbb{E}[X_{m,n}]$ with a recursive formula based on a feasible point selection process; 
and (iv) refine both estimations with a correction term.

\subsubsection{Property of the optimal removal}
\label{subsec: opt_removal}

We first show a necessary condition for the removed points to be optimal: 
the $k$-th removed point must have a net supply value of $k$ on the original curve $S(x; I)$. This is stated in the following proposition.  

\begin{prop}
\label{prop: S_eq_k}
$S(x_{v^*_k}; I) = k, \forall k \in \{1,\dots, n-m\}. $
\end{prop}

\begin{proof}
To show the proposition holds, it is sufficient to show the following claim is true: 
For any point $v'_k \in V'$, if $S(x_{v'_k}; I) < k$ or $S(x_{v'_k}; I) > k$, we can always swap a point in $V'$ with another point in $V\setminus V'$ to reduce $A(x; I\setminus V^{'})$; hence, $V' $ cannot be optimal. 

We begin with the case when $S(x_{v'_k}; I) < k$. According to Equation \eqref{eq: post_removal_S_x}, 
$v'_k$ must have a negative net supply value on the post-removal curve; i.e.,
$S(x_{v'_k}; I\setminus V') 
< 0$. 
Figure \ref{fig: necessary condition} (a) shows an example point $v'_k$ (indicated by the cross marker) in such a condition. A portion of the post-removal curve including the removal of this single point, is represented by the red dot-dash line. 
Now we may check the points in $I$ to the right hand side of $v'_k$ along the x-axis, until we encounter the first supply point $v\in V$. Such a supply point $v$ is guaranteed to be within $(x_{v'_k}, 1]$, otherwise $S(1; I\setminus V') \le S(x_{v'_k}; I\setminus V') < 0$, which violates Equation \eqref{eq: post_removal_S_x}.

Two cases may arise here for $v$, 
depending on whether it is in $V'$ or not.
The first case is when $v \notin V'$, as shown in Figure \ref{fig: necessary condition} (a). 
Since $v$ is the first supply point to the right of $v'_k$, the points within $(x_{v'_k}, x_v)$, if any, must all be demand points, as shown by the blue triangles in Figure \ref{fig: necessary condition} (a). 
As all demand points have negative supply values, the net supply values on the post-removal curve within $(x_{v'_k}, x_v)$ must be no larger than $S(x_{v'_k}; I\setminus V')$; i.e., for $x\in (x_{v'_k}, x_v)$, $S(x; I\setminus V') \le S(x_{v'_k}; I\setminus V') < 0$.  
Now swap $v'_k$ out of $V'$, and swap $v$ in. Note here such a point swap would only affect the post-removal curve within $[x_{v'_k}, x_v)$, and after the swap, the original post-removal curve, $S(x; I\setminus V')$, will increase by one unit within $[x_{v'_k}, x_v)$, while all the other parts remain the same. 
The area size under the curve is strictly reduced by the point swap, by an amount of 
$A(x; I\setminus V^{'})$, as shown by the gray area in Figure \ref{fig: necessary condition} (a). The reduced area size is: 
\begin{align*}
& \sum_{\{\forall i \in I, x_{v'_k} \leq x_i < x_v\}} l \cdot |S(x_i; I \setminus V' \setminus \{v\} \cup \{v'_k\}) - S(x_i; I\setminus V')| 
=\sum_{\{\forall i \in I, x_{v'_k} \leq x_i < x_v\}} l >0,
\end{align*}
The last inequality holds because $v$ must exist. 
Therefore, the claim is true for $S(x_{v'_k}; I) < k$ (i.e., $S(x_{v'_k}; I\setminus V') < 0$) and $v \notin V'$. 

The other case occurs when $v \in V'$. Since $v$ is the first supply point encountered in $V'$ to the right of $ v'_{k}$, $v$ must be $ v'_{k+1}$. 
Again, 
all points within $(x_{v'_k}, x_v)$, if any, must be demand points. 
In addition, since $v'_{k+1}$ itself is a removed point, $S(x_{v'_{k+1}}; I\setminus V') \le S(x_{v'_k}; I\setminus V') < 0 $. 
Then, we may simply change our focus from $v'_{k}$ to $v'_{k+1}$, 
and then scanning the points on the right side of $v'_{k+1}$, and continue until we find a point $v' \in V'$ with $S(x_{v'}; I\setminus V') < 0$, and its 
next supply point $v \notin V'$. 
The proof for the previous case would apply for swapping $v'$ and $v$. 
Hence, the claim is also true for $S(x_{v'_{k}}; I) < k$ and $v \in V'$. 

When $S(x_{v'_k}; I) > k$, the proof is symmetric. We look for points to swap that can lead to area reduction for $A(x; I\setminus V^{'})$ to the ``left" hand side of $v'_k$ along the x-axis, instead of the right. 
According to Equation \eqref{eq: post_removal_S_x}, $v'_k$ now must have a positive net supply value on the post-removal curve; i.e., $S(x_{v'_k}; I\setminus V') > 0$. 
A supply point to the left must exist within $[0, x_{v'_k})$, or otherwise $S(0; I\setminus V') \geq S(x_{v'_k}; I\setminus V') > 0$, which violates Equation \eqref{eq: post_removal_S_1}. There are two similar cases here depending on whether $v$ is or is not in $V'$. The logic of the proof is exactly symmetrical. The only difference is that, after swapping the two points, the original curve will ``decrease" by one unit, instead of increase, within the interval. The area reduction is still strictly greater than zero, 
as shown in Figure \ref{fig: necessary condition} (b).

We have shown that the claim is true in all possible conditions. 
This indicates that if any removed point $v'$ in an arbitrary set $V'$ does not satisfy $S(x_{v'_k}; I) = k$, then $V'$ cannot be optimal. 
Therefore, for any removed point $v^*_k$ in an optimal set $V^*$, $S(x_{v^*_k}; I) = k$ must hold necessarily. 
This completes the proof. 

\end{proof}

\begin{figure}[ht!] 
\centering 
\includegraphics[scale=0.6]{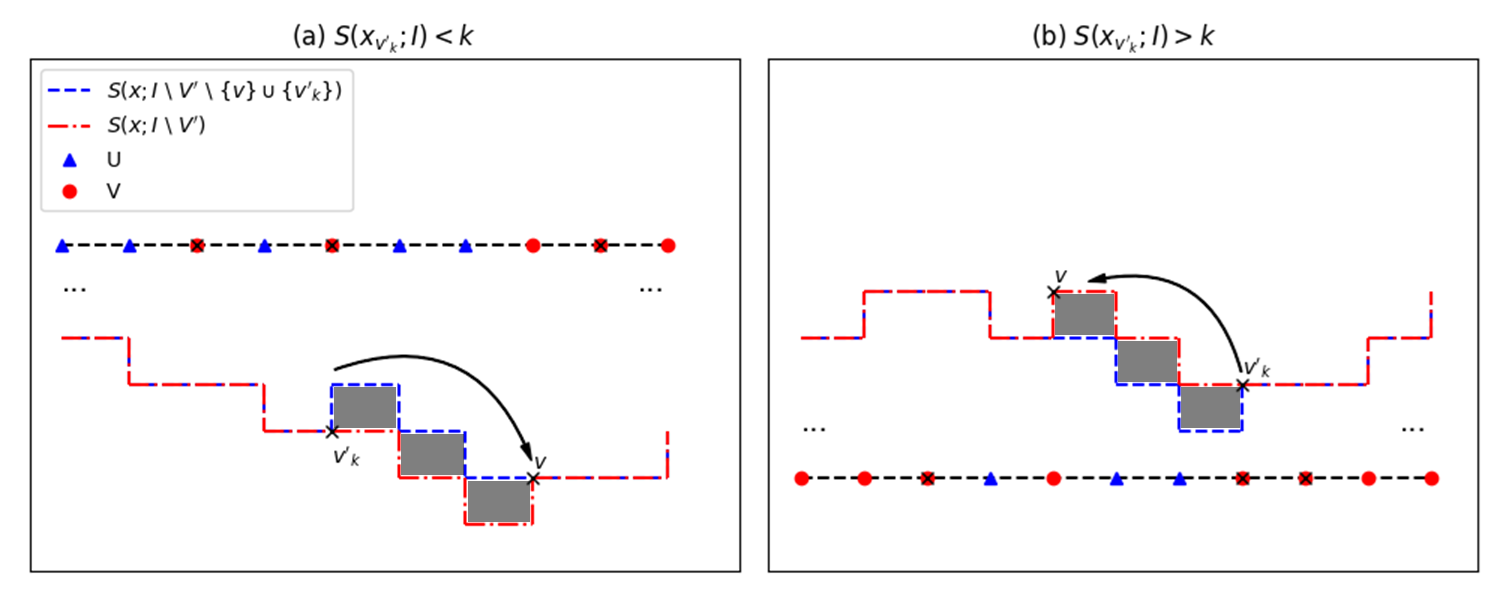}
\caption{Illustration of the point swap process}
\label{fig: necessary condition}
\end{figure}

Next, we present a useful property of the post-removal curve satisfying the optimality condition: 
$S(x_{v'_{k}}; I) = k, \forall k\in \{1, \dots, n-m\}$. 
According to Equation \eqref{eq: post_removal_S_x}, the net supply values on the post-removal curve at the locations of all removed points must be zero; i.e.: 
$S(x_{v'_{k}}; I\setminus V') = S(x_{v'_{k}}; I) - k = 0$. 
As such, the following proposition must hold.

\begin{prop}
\label{prop: balanced}
For any given $V'$, if $S(x_{v'_k}; I) = k, \forall k\in \{0, 1, \dots, n-m\}$, then each segment of the post-removal curve, $S(x; I\setminus V')$ within $(x_{v'_k}, x_{v'_{k+1}})$, can be regarded as the realized path of a corresponding balanced random walk.
\end{prop}


\subsubsection{An approximate closed-form formula}
\label{subsec: stars_and_bars}

Propositions \ref{prop: S_eq_k} and \ref{prop: balanced} show that the optimal post-removal curve $S(x; I\setminus V^*)$ contains $n-m+1$ segments with end points $\{x_{v^*_k}\}_{k = 0, 1, \cdots}$, and each segment $(x_{v^*_k}, x_{v^*_{k+1}})$ can be regarded as a realized path of a balanced random walk. 
Accordingly, we can decompose the optimal post-removal area $A(1; I\setminus V^*)$ into $n-m+1$ segments as follows:
\begin{align}
\label{eq: exp_Xnm_decomp}
&A(1; I\setminus V^*) = \sum_{k=0}^{n-m} \left[ A(x_{v^*_{k+1}}; I\setminus V^*) - A(x_{v^*_{k}}; I\setminus V^*) \right], 
\end{align}
where each term $A(x_{v^*_{k+1}}; I\setminus V^*) - A(x_{v^*_{k}}; I\setminus V^*)$ represents the area of a balanced random walk segment within range $(x_{v^*_k}, x_{v^*_{k+1}})$. 

For model convenience, we again ignore the impact of boundaries and assume that every point $v\in V$ is probabilistically identical and independent to be selected in $V^*$. 
Based on this assumption, we can directly apply Equation \eqref{eq: harel_unit_step} to estimate the expected area of each balanced random walk segment,
which shall be dependent on only the number of (balanced) points and the mean step size within each segment. 
Denoting $m_k$ as the number of demand (or supply points) in $U$ in the $k$-th segment, 
$\mathbb{E}[Z_{m,n}]$ can then be approximately estimated as the following sum: 
\begin{align}
\label{eq: exp_Xnm_decomp_2}
\mathbb{E}[Z_{m,n}] \approx  \sum_{k=0}^{n-m} \mathbb{E} \left [ l\cdot B(m_k) \right].
\end{align}
With our i.i.d. assumption for $V^*$ and disregard of the boundary effects, we must have $m_k, \forall k \in\{0, \dots, n-m\}$, to be i.i.d. This may result in an underestimation, because points at specific positions may have a higher probability of being selected in $V^*$ than others due to the presence of boundaries.

Then, we may only focus on the expected area of the first segment; i.e., when $k=0$. 
Let $\text{Pr}\{m_0=m'\}$
denote the probability that the first segment contains exactly $m' \in \{1, \cdots, m\}$ demand points, noticing that 
the total number of demand points across all $n-m+1$ segments must equal $m$; i.e., $\sum_{k=0}^{n-m} m_k = m$. 
To derive $\text{Pr}\{m_0=m'\}$
, one may relate it to the well-known ``stars and bars" combinatorial problem.
When we use $n-m$ ``bars" (points in $V^*$) to partition $m$ stars (all matched pairs of demand and supply points), 
there are $\binom{n}{n-m}$ possible combinations for such a partition. 
Once the first segment has been set to contain $m'$ stars, there remain $n-m'-1$ positions for the remaining $n-m-1$ bars to be placed, resulting in $\binom{n-m'-1}{n-m-1}$ possible combinations.
That is, 
\begin{align}
\label{eq: prob_m}
&\text{Pr}\{m_0=m'\} = \frac{\binom{n-m'-1}{n-m-1}}{\binom{n}{n-m}}.
\end{align}
Hence, following Equations \eqref{eq: exp_Xnn}, \eqref{eq: exp_Xnm_incomplete}, 
and \eqref{eq: prob_m}, and note $l=\frac{1}{n+m}$ in this case, we now have a closed-form formula for $\mathbb{E}[X_{m,n}]$, as follows:
\begin{align}
\label{eq: exp_Xnm}
\mathbb{E}[X_{m,n}] & \approx 
(n-m+1)\cdot\mathbb{E}\left[ l \cdot B(m_0) \right] \approx 
(n-m+1)\cdot l\cdot \sum_{m'=0}^{m} \text{Pr}\{m_0=m'\} \cdot B(m') \nonumber \\
&\approx \frac{n-m+1}{m(m+n)}\cdot 
\sum_{m'=0}^{m} \frac{\binom{n-m'-1}{n-m-1}}{\binom{n}{n-m}} \cdot \frac{m'2^{2m'-1}}{\binom{2m'}{m'}}.
\end{align}
{
\color{Black}
The presence of binomial coefficients in Equation (\ref{eq: exp_Xnm}) may lead to numerical overflow when $n$ and $m$ are very large. To address this, one can apply a logarithmic transformation to
all related terms and compute log-Gamma functions as a stepping stone; e.g., $\binom{n}{m} = \exp\left(\ln \binom{n}{m}\right) = \exp(\ln \Gamma (n + 1) -\ln\Gamma(m + 1)-\ln\Gamma(n-m+1))$. 
If we evaluate the binomial coefficients via Lanczos approximation \citep{lanczos1964gamma}, which incurs a constant time, then the computation of Equation (\ref{eq: exp_Xnm}) has a linear time complexity $\mathcal{O}(m)$.}

Again, this formula is approximate due to our simplifying assumptions. 
In the next section, we will present an alternative method to yield a more accurate estimation via a specific point removal and swapping process. 
 

 \subsubsection{Recursive formulas}
\label{subsec: recursive}
In theory, the optimal point removal for each realization shall be solved as a dynamic program (and possibly via the Bellman equation); however, such an approach is not suitable for closed-form formulas. 
Therefore, this section builds upon Propositions \ref{prop: S_eq_k} and \ref{prop: balanced} to propose a simpler point removal and swapping process that can yield a near-optimum point-removal solution for each realization. Then the area size estimation across realizations, based on this process, can be derived into a set of recursive formulas. The result provides a tight upper bound for $\mathbb{E}[X_{m,n}]$.

First, we propose a point removal process that is developed based on the necessary optimality conditions in Proposition \ref{prop: S_eq_k}, such that it provides a reasonably good (e.g., locally optimal) balanced random walk. 
We initialize $\hat{V} = \emptyset$ and $k=1$. Starting from $x = 0$, scan the supply points in $V$ from left to right along the x-axis and check if a point satisfies the following conditions: 
(i) it has a net cumulative supply value of $k$; 
(ii) the nearest neighbor (in $I$) on the left, if existing, has a net supply value of $k-1$; 
and (iii) the net supply values of all points (in $I$) on the right hand side of this point is greater than or equal to $k$. 
If conditions (i)-(iii) are all satisfied by a point, denoted $\hat{v}_{k}$, then add it to the current set $\hat{V}$,  
increase $k$ by 1, and repeat the above procedure until $k=n-m$. 
The net supply values of the points selected in this procedure are guaranteed to form an increasing sequence; i.e.: 

\begin{align*}
S(x_{\hat{v}_k}; I) = k, \quad \forall k\in \{1, \dots, n-m\}.     
\end{align*}
This is because, at step $k$ of the above process, the curve $S(x; I)$ must have a net supply value of $k-1$ at the previously selected point $\hat{v}_{k-1}$, and a value of at least $n-m$ at the final step at $x = 1$. Hence, from intermediate value theorem, point $\hat{v}_{k}$ can always be found in $(x_{\hat{v}_{k-1}}, 1]$. 
Note that the removal of $\hat{v}_k$ at step $k$ only affects the curve on its right-hand side, with the segments on its left-hand side being balanced random walks. 
From the construction of the process, the post-removal cumulative net supply curve $S(x; I\setminus \hat{V})$ remains non-negative in $(x_{\hat{v}_k}, 1]$ at every step; i.e.,  
\begin{align*}
S(x; I\setminus \hat{V}) \geq 0, \quad \forall x \in (x_{\hat{v}_k}, 1], \text{ } k\in \{1,\cdots,n-m\},    
\end{align*}
which also implies that 
$S(x;I\setminus\hat V)\geq 0$ for all $x \in (x_{\hat{v}_1}, 1]$ throughout the entire process.
Figure \ref{fig: feasible_removal} (a) shows a simple example of the point removal process, with $n-m = 2$. 
The original curve, $S(x; I)$, is represented by the red dot-dash line, and two selected removal points $\hat{v}_1$ and $\hat{v}_2$ are represented by the black cross markers. 
At step 1 and 2, the removal of $\hat{v}_1$ and $\hat{v}_2$ decrease the net supply values of points in the red and purple shaded regions, respectively. 

\begin{figure}[ht!] 
\centering 
\includegraphics[scale=0.65]{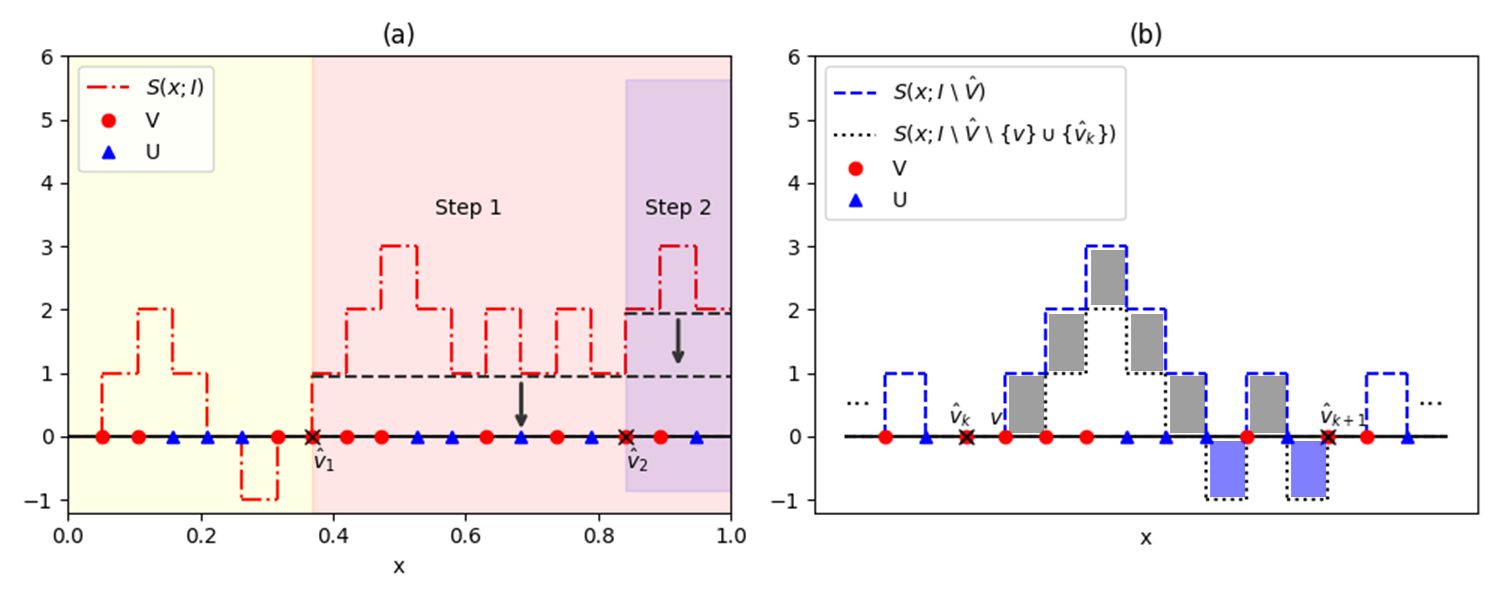}
\caption{Illustration of the point removal and swap procedure.}
\label{fig: feasible_removal}
\end{figure}

We again use the term ``segment of the curve" but now refer to the post-removal curve $S(x; I\setminus \hat{V})$ within $(x_{\hat{v}_{k}}, x_{\hat{v}_{k+1}})$ as the ``$k$-th segment" (which must be balanced). 
From the perspective of each point $\hat{v}_k$, the post-removal area in $(x_{\hat{v}_{k}}, 1]$ 
includes those within the $k$-th segment and $(x_{\hat{v}_{k+1}}, 1]$, both of which depend on the length of the $k$-th segment. 
Let $\text{Pr}\{\hat{m}_k \mid a\}$ 
denote the probability for the $k$-th segment to contain $\hat{m}_k$ demand (or supply) points, when there are exactly $a$ demand points in 
$(x_{v_{k}}, 1]$.
In the step to select $v_{k+1}$, there are $n-m-k$ supply points to be removed; thus, 
there are $a+n-m-k$ supply points in $(x_{v_{k+1}}, 1]$.
That is, the original curve $S(x;I)$ 
starts from value $k$ at $\hat{v}_{k}$ and returns to $k$ after exactly $2\hat{m}_k$ steps, 
which occurs with probability 
\begin{align*}
\binom{a}{\hat{m}_k}\binom{a+n-m-k}{\hat{m}_k}/\binom{2a+n-m-k}{2\hat{m}_k},
\end{align*}
but never returns to value $k$ afterwards
, with probability from the well-known Ballot's theorem \citep{Addario_2008_Ballot}:
\begin{align*}
\frac{n-m-k}{2a+n-m-k-2\hat{m_k}}.    
\end{align*}

As such, we have 
\begin{align}
\label{eq: prob_0}
\text{Pr}\{\hat{m}_k \mid a\}
= \frac{\binom{a}{\hat{m}_k}\binom{a+n-m-k}{\hat{m}_k}}{\binom{2a+n-m-k}{2\hat{m}_k}} \cdot \frac{n-m-k}{2a+n-m-k-2\hat{m}_k}.
\end{align}

Let $\hat{Z}_{k, a}$ represent the area enclosed by the x-axis and the post-removal curve $S(x;I\setminus \hat{V})$ within $(x_{\hat{v}_k}, 1]$, which contains exactly $a$ demand points. 
From Equation \eqref{eq: harel_unit_step}, the expected areas size of the balanced random walk inside the $k$-th segment $(x_{\hat{v}_{k}}, x_{\hat{v}_{k+1}})$ is simply:
\begin{align*}
l \cdot B(\hat{m}_k).     
\end{align*}
Hence, 
conditional on $\hat{m}_k$, we have
\begin{align}
\label{eq: exp_Z_ab_wo_swap}
\mathbb{E}[\hat{Z}_{k, a}] 
&= \sum_{m'=0}^{a} \text{Pr}\{\hat{m}_k=m' \mid a \}\cdot\left[ l \cdot B(m') + \mathbb{E}[\hat{Z}_{k+1, a-m'}]\right], \nonumber \\
& \forall k \in \{0, \dots, n-m-1\}, \quad 0 \le a \le m,
\end{align}
and when $k=m-n$, 
\begin{align*}
\mathbb{E}[\hat{Z}_{n-m, a}] = l \cdot B(a), \quad \forall 0 \le a \le m.
\end{align*}
When $k=0$, $\mathbb{E}[\hat{Z}_{0,m}]$ represents the expected area of the entire post-removal curve. 

Next, we propose a refinement process based on local point swapping. It is intended to further reduce the post-removal area for a more accurate upper bound estimate.  
For all $k \in \{1, \dots, n-m-1\}$,\footnote{Recall that all points in these segments have ``non-negative" post-removal curve values. We cannot perform such a point swap on the $(n-m)$-th segment because $\hat{v}_{n-m+1} \notin \hat{V}$. } if there exists a point $v \in V$ 
that satisfies (i) $x_{\hat{v}_{k}} < x_{v} < x_{\hat{v}_{k+1}}$, and (ii) 
$v$ is the nearest neighbour of $\hat{v}_{k}$ on the right, then we 
swap $\hat{v}_{k+1}$ out of set $\hat{V}$ and swap point $v$ in. 
We propose to perform exactly one such point swap\footnote{Although additional point swaps could be performed, we limit the process to one swap per segment in order to maintain the model’s simplicity and tractability. Moreover, the marginal benefits of additional swaps are expected to diminish.} to each of the segments. 

An example of point swap is illustrated in Figure \ref{fig: feasible_removal} (b). 
The post-removal curve $S(x; I\setminus \hat{V})$ is represented by the blue dashed curves.
After a swap between points $\hat{v}_{k+1}$ and $v$, the $k$-th curve segment will be shifted down by one unit (while all other segments remain the same), as indicated by the black dotted curves. 
Such a point swap can result in both reductions (if net supply $S(x; I\setminus \hat{V}) >0$ before the swap, as the grey rectangles) and additions (if net supply $S(x; I\setminus \hat{V})=0$ before the swap, as the blue rectangles) to the enclosed area size. 
The following proposition says that the expected total area size will be reduced for each of the swaps. This indicates that the proposed point swapping process will yield a smaller (or at least equal) expected post-removal area size.

\begin{prop} 
Suppose $\hat{m}_{k,0} \le \hat{m}_k$ is the random number of points with zero net supplies within the $k$-th segment. The expected area reduction in the $k$-th segment is
\begin{equation} 
l \cdot \left(
2\hat{m}_k - 2\mathbb{E}[\hat{m}_{k,0} \mid \hat{m}_k]
\right), \quad
\forall k \in \{1, \dots, n-m-1\},
\label{area-reduction}
\end{equation}
and it is nonnegative.
\begin{proof}
According to the point swapping process, 
the area size reduction can be computed as the difference between the number of positive net-supply points, $2\hat{m}_k-\hat{m}_{k,0}$, and that of the zero net-supply points $\hat{m}_{k,0}$, multiplied by the step size $l$.  
Thus, the expectation of the area change with respect to $\hat{m}_{k,0}$, conditional on $\hat{m}_k$, is simply given by the left-hand side of Equation \eqref{area-reduction}. We next show that this term must be non-negative.

To compute $\mathbb{E} [\hat{m}_{k,0} \mid \hat{m}_k]$, it is equivalent to compute the expected number of times a random walk with $2\hat{m}_k$ steps returns to zero. 
Clearly, this expectation equals zero when $\hat{m}_k = 0$. Meanwhile, \cite{Harel_randomWalkArea_1993} derived the probability that a random walk with $2\hat{m}_k$ steps returns to zero at the $2j$-th step, $\forall j \in \{1, \dots, \hat{m}_k\}$: 
\begin{align*}
\binom{2j-1}{j} \binom{2\hat{m}_k-2j}{\hat{m}_k-j} / \binom{2\hat{m}_k-1}{\hat{m}_k}
, \quad \forall k \in \{1, \dots, n-m-1\}
.    
\end{align*}
Summing these probabilities across all possible values of $j$, 
we have
\begin{equation}
\label{eq: exp_mk0}
\mathbb{E}[\hat{m}_{k,0}\mid \hat{m}_k] = 
\begin{cases}
0, \quad &\forall \hat{m}_{k} = 0,\\
\sum_{j=1}^{\hat{m}_k} \binom{2j-1}{j} \binom{2\hat{m}_k-2j}{\hat{m}_k-j} / \binom{2\hat{m}_k-1}{\hat{m}_k}, \quad &\forall \hat{m}_{k} \neq 0.
\end{cases}
\
\end{equation}
When $\hat{m}_k\neq 0, $ each term inside the summation represents a probability and must be less than or equal to 1. Hence, the entire summation must be less than or equal to $\hat{m}_k$. As a result, the expected reduced area given by Equation \eqref{area-reduction} must be non-negative.
\end{proof}
\end{prop}

Now, adding 
\eqref{area-reduction} as a correction term into 
\eqref{eq: exp_Z_ab_wo_swap}, and note $l = \frac{1}{n+m}$, we have
\begin{equation}
\begin{aligned}
\label{eq: exp_area_tail_swap}
&\mathbb{E}[\hat{Z}_{0, m}] = \sum_{m'=0}^{m} \text{Pr}\{\hat{m}_0= m' \mid m\}\cdot \left[ l \cdot B(m') + \mathbb{E}[\hat{Z}_{1, m-m'}]\right], \\
&\mathbb{E}[\hat{Z}_{k, a}] = \sum_{m'=0}^{a} \text{Pr}\{\hat{m}_k=m' \mid a \}\cdot\left[l \cdot B(m') -  l \cdot (2m' - 2\mathbb{E} [\hat{m}_{k,0} \mid m']) + \mathbb{E}[\hat{Z}_{k+1, a-m'}]\right], \\
& \qquad \qquad \qquad \qquad \qquad \qquad \qquad \qquad \qquad \qquad \qquad \qquad \qquad \quad
\forall k \in \{1, \dots, n-m-1\},\\
&\mathbb{E}[\hat{Z}_{n-m, a}] = l \cdot B(a).
\end{aligned}
\end{equation}
Together with \eqref{eq: prob_0} and \eqref{eq: exp_mk0}, all the above expected area sizes can be solved recursively. The expected matching distance $\mathbb{E}[X_{m,n}]$ is related to $\mathbb{E}[\hat{Z}_{0,m}]$, as follows:
\begin{align}
\label{eq: exp_Xnm_upper}
&\mathbb{E}[X_{m,n}] \approx \frac{1}{m}\cdot \mathbb{E}[\hat{Z}_{0,m}].
\end{align}
\textcolor{Black}{Again, all binomial coefficients in the above formulas can be computed accurately via log-Gamma functions and Lanczos approximation in constant time. Then, the recursive computation of this formula has a polynomial time complexity: $\mathcal{O}((n-m)m^2)$.}

\section{1D RBMP Variants}

The results in Section \ref{sec: model} hold for the one-dimensional RBMP problem on a lattice. This section introduces three one-dimensional problem variants that serve as building blocks towards network problems.

\label{sec: variant}
\subsection{Periodic Boundary}
\label{subsec: periodic}
{
In this section, we study a problem variant with periodic boundaries
~\citep{Boniolo_2014}, where the two sets of points in $U$ and $V$ are evenly distributed along a unit-length ``ring" instead of a line segment. 
Here we only consider the case where the two sets contain an equal and sufficiently large number of points, and let $X^\circ_{n, n}$ denote the average optimal matching distance. 

A problem instance on the ring can be seen as being ``bent" from interval $[0,1]$ such that points $x=0$ and $x=1$ overlap. The cumulative net supply curve can be constructed for any coordinate $x$ along the ring and any subset of points $I' \subseteq I$; i.e., $S(x; I') = \sum_{\{\forall i\in I', x_i \leq x\}} z_i$. 
Figure \ref{periodic} shows an example of one problem instance, where the red dots and blue triangles represent points from $V$ and $U$, respectively, and the red dot-dash line represents the cumulative net supply curve $S(x;I)$. 

\begin{figure}[H]
    \centering
    \includegraphics[width=0.6\linewidth]{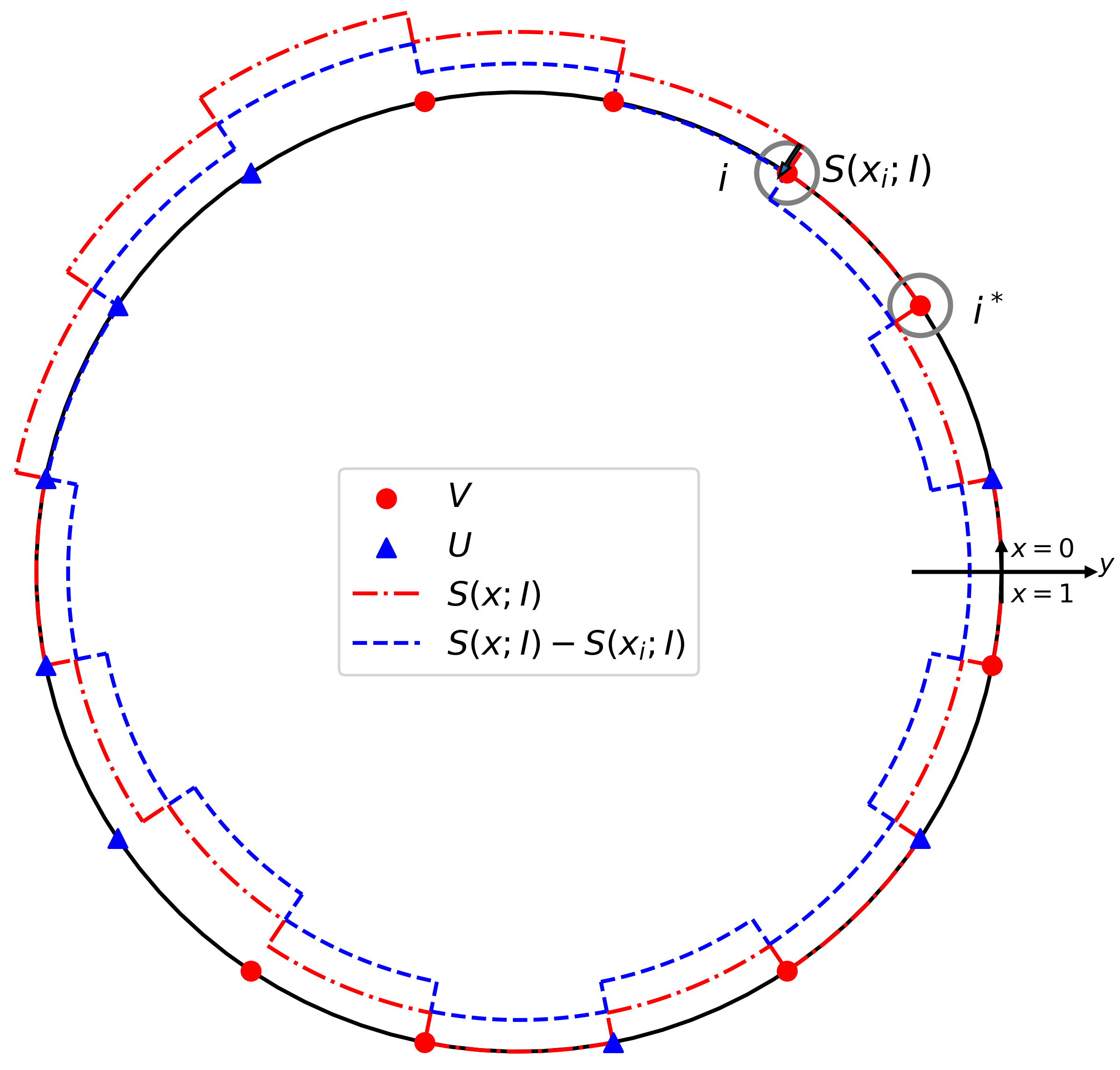}
    \caption{Cumulative net supply curve with periodic boundaries.}
    \label{periodic}
\end{figure}

We next show a special property of the optimal matching based on \cite{werman_1986}. 
As different cumulative net supply curves (and hence candidate matching solutions) can be constructed by shifting the x-axis along the ring, they showed that the optimal matching is achieved by identifying the optimal shift that minimizes the area under the resulting curve. 
For any point $i\in I$, if we had constructed a new cumulative net supply curve starting immediately after $x_i$, this curve is given by $S(x;I) - S(x_{i};I)$. 
The blue dashed line in Figure \ref{periodic} shows an example. 
The optimal matching distance is given by the minimum area under all shifted curves over all points in $I$; i.e., 
\begin{align}
\min_{\{i\in I\}} \sum_{\{i'\in I\}} l\cdot |S(x_{i'};I)- S(x_i;I)|. 
\label{periodic_min}
\end{align}

This can further lead to the following optimality conditions. 
Let $i^*$ denote the optimal solution to Equation \eqref{periodic_min} and let $S(x_{i^*};I) = K \in \mathbb{Z}$. $S(x;I)-K$ represents the optimal cumulative net supply curve whose area equals the minimal total matching distance. 
Any other cumulative net supply curve constructed by changing $i^*$ to a neighboring point $i'$, where $S(x_{i'};I) = K\pm1$, can not decrease the area size. 
When $S(x_{i'};I) = K+1$, the curve $S(x;I) - K$ will be shifted downward by one unit. The area change by such a shift can be computed as the total length of the ring with non-positive net supply minus that with positive net supply. 
This area change is non-negative and the following condition must hold.
\begin{align}
-\sum_{\{i\in I:S(x_i;I)-K>0\}}l + 
\sum_{\{i\in I:S(x_i;I)-K\leq 0\}}l  \geq 0.
\label{periodic_cond_1}
\end{align}
Similarly, if we choose $i'$ such that $S(x_{i'}; I) = K-1$, the area change can be computed as the total ring length with non-negative net supply minus that with negative net supply. Thus, 
\begin{align}
\sum_{\{i\in I:S(x_i;I)-K\geq 0\}}l - 
\sum_{\{i\in I:S(x_i;I)-K< 0\}}l  \geq 0.
\label{periodic_cond_2}
\end{align}

When $n$ is sufficiently large, the expected total ring length with zero net supply tends to zero; i.e.,

\begin{align}
\mathbb{E}\left[\sum_{\{i\in I:S(x;I) - K=0\}}l\right] \xrightarrow{n\to \infty} 0.
\label{periodic_cond_3}
\end{align}
The detailed proof can be found in \ref{apx: optimal}. The intuition behind Equation \eqref{periodic_cond_3} is that the cumulative net supply curve on the ring corresponds to a random walk with $2n$ steps, and the left-hand-side of Equation \eqref{periodic_cond_3} equals the expected number of times the random walk reaches the value of $K$, which diminishes with $n$ asymptotically. 

Then, we take the expectations of Equations \eqref{periodic_cond_1} - \eqref{periodic_cond_2} across random walk realizations. Their difference, after subtracting \eqref{periodic_cond_3}, gives: 
\begin{align}
\mathbb{E}\left[\sum_{\{i\in I:S(x_i;I)-K>0\}}l\right] \xrightarrow{n\to \infty} 
\mathbb{E}\left[\sum_{\{i\in I:S(x_i;I)-K<0\}}l\right].
\label{periodic_cond_4}
\end{align}
It indicates that 
the expected total ring lengths with positive and negative net supplies are (approximately) equal if we draw the cumulative net supply curve based on the optimal solution of Equation \eqref{periodic_min}. Due to symmetry, the expected areas under the curve with positive and negative net supply are also equal. 

This property can be used to estimate the 
the optimal matching distance. 
We can use $B(n/2)$, the expected area size under a $n$-step balanced random walk, to approximate the expected area under the curve with either positive or negative net supplies.
Therefore, the total expected area under the optimal cumulative net supply curve is approximately $2B(n/2)$. Noting $l = \frac{1}{2n}$, the expected average matching distance can be computed as:
\begin{align}
\mathbb{E}[X_{n,n}^\circ] \approx \frac{l}{2n}\cdot2B(n/2) = \frac{1}{4\sqrt{2}}\sqrt{\frac{\pi}{n}}, \quad n\to\infty.
\label{eq: exp_Xnn_periodic}
\end{align}

}

\subsection{Uniform Point Distribution}
\label{subsec: correction}


In this section, we release the two sets of points from the lattice but assume that they 
\textcolor{Black}{are independently and uniformly distributed}
anywhere along the unit-length line segment. Let $X^\text{u}_{m, n}$ denote the average optimal matching distance in this case. 
The distance from each point $i\in I$ to its next neighbor is now a random variable, denoted $l_i, \forall i\in I \setminus \{|I|\}$,  
\textcolor{Black}{
and $\mathbb{E}[l_i] = \frac{1}{n+m+1}, \forall i$.} 
Under such an assumption, Propositions \ref{prop: S_eq_k} and \ref{prop: balanced} still hold, as long as we use $l_i$ to replace $l$ everywhere and slightly adjust the expression of the reduced area size in the proofs; e.g.,
\begin{align*}
& \sum_{\{\forall i \in I, x_{v'_k} \leq x_i < x_v\}} l_i \cdot |S(x_i; I \setminus V' \setminus \{v\} \cup \{v'_k\}) - S(x_i; I\setminus V')| 
=\sum_{\{\forall i \in I, x_{v'_k} \leq x_i < x_v\}} l_i >0.
\end{align*}

The distance formulas for the lattice problem, such as Equations \eqref{eq: exp_Xnn}, \eqref{eq: exp_Xnm} and \eqref{eq: exp_Xnm_upper}, can be used as a first-order approximation. 
However, by ignoring the randomness of $\{l_i\}$, they tend to overestimate the distances, because intuition tells us that the matching distance shall decrease when the points cluster. 
On a lattice, the random walk's step size $l$ 
represents the minimum 
achievable matching distance between any two points. 
For unbalanced lattice problems, as $n\rightarrow \infty$, the estimated optimal matching distance from both Equations \eqref{eq: exp_Xnm} and \eqref{eq: exp_Xnm_upper} should approach $l$ due to the presence of the lattice. 
However, when $\{l_i\}$ varies, those successfully matched points from $V$ are more likely to have shorter distances to their neighbors.

As a result, after removing these unmatched points, the ``true" expected step size of the balanced random walk segments, which contain only the successfully matched points, should be no larger than $l$. 
Therefore, Equations \eqref{eq: exp_Xnn}, \eqref{eq: exp_Xnm} and \eqref{eq: exp_Xnm_upper} may lead to an overestimation of the ``true" optimal matching distance, and the resulting estimation gap is expected to be positively related to the number of unmatched points $n-m$.
{\color{Black}
We formally prove this intuition in the following proposition. 

\begin{prop}
\label{prop: uniform_vs_lattice}
$\mathbb{E}[X_{m,n}^\text{u}] \leq \mathbb{E}[X_{m,n}]$, and the equality holds when $n = m$.
\end{prop}
\begin{proof}
Here we denote the total post-removal area of one problem instance with uniform distribution by adding superscript $\text{u}$, $A^\text{u}(1; I\setminus V')$, and it can be written as:
\begin{align*}
A^\text{u}(1; I\setminus V') = \sum_{\{\forall i\in I \setminus \{|I|\}\}} l_i\cdot |S(x_{i}; I\setminus V')| = \langle\mathbf{l},\mathbf{s}(V')\rangle,
\end{align*}
where $\textbf{l} = (l_i)_{i\in I\setminus\{|I|\}}$ denotes the vector of step sizes; $\textbf{s}(V') = (|S(x_i; I\setminus V')|)_{i\in I\setminus\{|I|\}}$ denotes the vector of the absolute values of post-removal cumulative net supply; and $\langle\cdot, \cdot\rangle$ is the inner product operator. 
Accordingly, the optimal total matching distance is obtained by finding the optimal $V^* \subset V$ that minimizes the total post-removal area:
\begin{align*}
A^\text{u}(1; I\setminus V^*) = \min_{V'\subset V} A^\text{u}(1;I\setminus V') = \min_{V'\subset V} \langle \textbf{l}, \textbf{s}(V')\rangle.
\end{align*}

We will next show that 
$\min_{V'\subset V} \langle \textbf{l}, \textbf{s}(V')\rangle$ is a concave function with respect to \textbf{l}. Suppose $\textbf{l}_1, \textbf{l}_2$ are two arbitrary realized instances of step size vectors. With any $\delta \in [0, 1]$, we have 
\begin{align*}
\min_{V'\subset V} 
\langle \delta \cdot \textbf{l}_1 + 
(1-\delta) \cdot \textbf{l}_2, 
\textbf{s}(V')
\rangle
&=
\min_{V'\subset V} \left[
\delta\cdot \langle \textbf{l}_1, 
\textbf{s}(V')
\rangle
+
(1-\delta) \cdot \langle \textbf{l}_2, 
\textbf{s}(V')
\rangle \right]\\
&\geq
\delta\cdot \min_{V'\subset V} \langle \textbf{l}_1, 
\textbf{s}(V')
\rangle
+
(1-\delta) \cdot \min_{V'\subset V} \langle \textbf{l}_2, 
\textbf{s}(V')
\rangle.
\end{align*}
Hence, concavity holds. Per Jensen's inequality, we have 
\begin{align*}
\mathbb{E}[X_{m,n}^\text{u}] = \mathbb{E}\left[\min_{V'\subset V} \langle \textbf{l}, 
\textbf{s}(V')\right]
\leq
\min_{V'\subset V} \langle \mathbb{E}[\textbf{l}], 
\textbf{s}(V')
\rangle  =\mathbb{E}[X_{m,n}].
\end{align*}

In the special case when $n = m$, there is no excessive supply points to be removed, and we have the first equality of the following:
\begin{align*}
\mathbb{E}[A^\text{u}(1; I)] &=\mathbb{E}\left[\sum_{\{\forall i\in I \setminus \{|I|\}\}} l_i\cdot |S(x_{i}; I)|\right]
=\mathbb{E}[l_i] \cdot \mathbb{E}\left[\sum_{\{\forall i\in I \setminus \{|I|\}\}} |S(x_{i}; I)|\right]\\ 
&= l \cdot \mathbb{E}\left[\sum_{\{\forall i\in I \setminus \{|I|\}\}} |S(x_{i}; I)|\right]
= \mathbb{E}\left[\sum_{\{\forall i\in I \setminus \{|I|\}\}} l\cdot |S(x_{i}; I)|\right] = \mathbb{E}\left[A(1;I)\right].
\end{align*}
The second equality holds because $l_i$ is independent of $|S(x_i;I)|$ for each $i$, while the rest are simple algebraic manipulations. 
This completes the proof.
\end{proof}
}


It is non-trivial to estimate this gap explicitly. Therefore, we introduce an approximate correction term derived based on the case when $n\gg m$. 
Recall 
\cite{shen_zhai_ouyang_2024} proposed a set of formulas and asymptotic approximations for estimating the matching distance of RBMPs in any dimensions. Specifically for the one-dimensional case, they have:
\begin{align}
\label{eq: general_model}
&\mathbb{E}[X^{\text{u}}_{m,n}] 
\xrightarrow{n\gg m} \frac{1}{2n}. 
\end{align}
Thus, when $n \gg m$, the gap between the estimations given by Equations \eqref{eq: exp_Xnm} or \eqref{eq: exp_Xnm_upper} and the one given by Equation \eqref{eq: general_model} is simply 
\textcolor{Black}{
$l - \frac{1}{2n} = \frac{n-m+1}{2n(m+n+1)}$. 
Note that when $m<n$, this gap is positive, and when $m=n$, this gap converges to zero, which is consistent with the property stated in Proposition \ref{prop: uniform_vs_lattice}. }
We will use this as an approximate correction term and subtract it from the estimates provided by both Equations \eqref{eq: exp_Xnm} and \eqref{eq: exp_Xnm_upper} across all $m$ and $n$ values.  
The final formulas for estimating $\mathbb{E}[X^\text{u}_{m,n}]$ after applying the correction are as follows:
\begin{align}
\label{eq: exp_Xnm_corrected}
\mathbb{E}[X^\text{u}_{m,n}] 
&\approx \left[\frac{n-m+1}{m(m+n)}\cdot 
\sum_{m'=0}^{m} \frac{\binom{n-m'-1}{n-m-1}}{\binom{n}{n-m}} \cdot \frac{m'2^{2m'-1}}{\binom{2m'}{m'}}\right] - 
{\color{Black}
\frac{n-m+1}{2n(m+n+1)}},\\
\label{eq: exp_Xnm_upper_corrected}
&\mathbb{E}[X^\text{u}_{m,n}] \approx \left[\frac{1}{m}\cdot \mathbb{E}[\hat{Z}_{0, m}]\right] - 
{\color{Black}
\frac{n-m+1}{2n(m+n+1)}},
\end{align}
where $\mathbb{E}[\hat{Z}_{0, m}]$ is given by Equation \eqref{eq: exp_area_tail_swap}.

\subsection{Arbitrary-length Line}
\label{subsec: scaling}
This section builds upon Section \ref{subsec: correction} but further 
allows the line segment to have an arbitrary length of $L$ distance units (du). 
In addition, we introduce point densities (per du) $\mu$ and $\lambda$, where $\mu \le \lambda$, such that $m = \mu L$ and $n = \lambda L$. 
The average distance between any two adjacent points $l=\frac{1}{\lambda+\mu}$ du. 
The average optimal matching distance in this case, denoted as $X_{\text{E}}$ for an ``edge", should be governed by three parameters: $\mu$, $\lambda$, and $L$. 

We will consider a few possible cases. For the balanced case $(\lambda=\mu)$, 
we can simply substitute $n = \lambda L$ 
and $l=\frac{1}{2\lambda}$ 
into Equations \eqref{eq: exp_Xnn}, which gives an updated expected matching distance as: 
\begin{align}
\label{eq: exp_Xnn_scale}
\mathbb{E}[X_{\text{E}}] 
= \frac{l\cdot B(\lambda L)}{\lambda L}
= \frac{2^{2\lambda L-1}}{2\lambda \binom{2\lambda L}{\lambda L}} 
\xrightarrow{\lambda L\gg 1} \frac{1}{4}\sqrt{\frac{\pi L}{\lambda }}, 
\quad \text{if } 
\lambda =\mu.
\end{align}
It can be observed from this formula that the expected matching distance  scales with $\sqrt{L}$.
For the highly unbalanced cases $(\lambda \gg \mu)$, 
we can scale the asymptotic approximation in Equation \eqref{eq: general_model} by multiplying $L$ and substitute $n = \lambda L$, which gives the following:
\begin{align}
\label{eq: exp_general_scale}
\mathbb{E}[X_{\text{E}}] 
\approx \frac{L}{2\lambda L} = \frac{1}{2\lambda}, 
\quad \text{if } 
\lambda \gg \mu.
\end{align}
In this case, the formula shows that the expected matching distance is independent of $L$. 
For other unbalanced cases ($\lambda \gtrapprox  \mu$), we may substitute $n = \lambda L$, $m = \mu L$,  $l=\frac{1}{\lambda+\mu}$, and the correction term $l - \frac{1}{2\lambda}$ into Equations \eqref{eq: exp_Xnm_corrected}-\eqref{eq: exp_Xnm_upper_corrected}, which leads to the following: 
\begin{align}
\label{eq: exp_XE_corrected}
\mathbb{E}[X_{\text{E}}] 
&\approx \frac{\lambda L -\mu L+1}{\lambda+\mu}\cdot \sum_{m'=0}^{\mu L} \text{Pr}\{m_0=m'\} \cdot B(m') - \frac{\lambda-\mu}{2\lambda(\mu+\lambda)}, \quad \text{if } 
\lambda \gtrapprox  \mu,  \text{ or }\\
\label{eq: exp_XE_upper_corrected}
&\mathbb{E}[X_{\text{E}}] \approx \frac{1}{\mu L}\cdot \mathbb{E}[\hat{Z}_{0, \mu L}] - \frac{\lambda-\mu}{2\lambda(\mu+\lambda)}, \quad \text{if } 
\lambda \gtrapprox \mu.
\end{align}
These formulas do not directly tell how the expected matching distance scales with $L$. However, our numerical experiments show 
that when $\frac{\lambda}{\mu} \approx 1$, the distance formula behaves more similarly to the balanced case and increases monotonically with $\sqrt{L}$. As $\frac{\lambda}{\mu}$ increases, the formula value quickly converges to a fixed value, regardless of $L$. 

The logic behind this somewhat counter-intuitive property is that when one subset of vertices is dominating over the other, the vertices in the dominated subset are more likely to find matches locally; i.e., they only interact with the points nearby, so the expected matching distance does not increase with the length of the line. However, when the numbers of vertices in both subsets are similar, especially when they are equal, the optimal point matches tend to be found globally and they are less independent of one another; i.e., some points need to be matched with other points across the entire line. This is exactly the ``correlation" issue that was discussed in \citet{mezard_euclidean_1988}. 
Further discussion on the scaling properties of the expected matching distance in higher-dimension continuous spaces can be found in \citet{shen_zhai_ouyang_2024}. 


\section{Network RBMP Estimators} 
\label{sec: network model}
Now we are ready for the matching problem in a network.  
Let $\text{G} = (\mathcal{V}, \mathcal{E})$ be an undirected graph with node set $\mathcal{V}$ and edge set $\mathcal{E}$.  
On each edge $e \in \mathcal{E}$, the 
point subsets $U_e$ and 
$V_e$ are generated according to homogeneous Poisson processes, 
where $m_e=|U_e|$ and $n_e=|V_e|$ are now random variables. 
Matching is now conducted between the two sets of points on all edges: $U = \bigcup_{e\in \mathcal{E}} U_e$ and $V = \bigcup_{e\in \mathcal{E}} V_e$. 
The average optimal matching distance in this case, denoted as $X_{\text{G}}$ for a ``graph", 
should be influenced by the graph's topology. 
In this paper, we focus on a special type of graphs with the following properties: (i) all nodes in $\mathcal{V}$ have the same degree, $D$, and hence the graph is $D$-regular; 
(ii) all edges in $\mathcal{E}$ have the same length $L$; and (iii) the Poisson densities, $\mu$ and $\lambda$, are respectively identical across all edges.
Note that nodes and edges can be distributed in 2, 3- or any higher dimensional space.
Since all edges are translationally symmetric, we can start from one arbitrary edge, and study how $X_{\text{G}}$ is determined by four key parameters: $\mu$, $\lambda$, $L$ and $D$. 
Here, we further focus on graphs with $D\geq 3$, because the graph with $D=2$ reduces to a ring which has been studied in Section \ref{subsec: periodic}.

We first categorize the matches occurring on an edge $e$ into two types. 
We say an arbitrary point $u\in U_e$ is ``locally" matched if its corresponding match point $v$ is on the same edge, with matching distance $X_{\text{G}}^\text{l}$; otherwise, it is ``globally" matched, with matching distance $X_{\text{G}}^\text{g}$.
Let $\alpha$ represent the probability for global matching, then from 
the law of total expectation, 
the expected distance $\mathbb{E}[X_{\text{G}}]$ can be expressed as follows:

\begin{align}
\label{eq: X_G}
\mathbb{E} [X_{\text{G}}] = 
(1-\alpha)\cdot \mathbb{E}[X_{\text{G}}^{\text{l}}] + \alpha\cdot \mathbb{E}[X_{\text{G}}^\text{g}].  
\end{align}
We approximate the local matching distance $\mathbb{E}[X_{\text{G}}^\text{l}]$ from Equations \eqref{eq: exp_Xnn_scale}-\eqref{eq: exp_general_scale} as if it were from an arbitrary-length line under parameters $\mu$, $\lambda$, $L$, i.e.,
\begin{align}
&\mathbb{E}[X_{\text{G}}^\text{l}] 
\approx 
\mathbb{E}[X_{\text{E}}].
\end{align}

Next, to estimate $\alpha$ and $\mathbb{E}[X_{\text{G}}^\text{g}]$, we propose a feasible process that is expected to generate a reasonably good matching solution for every realization. 
It prioritizes local matching and works as follows. 
For each edge $e\in \mathcal{E}$, if $n_e \ge m_e$, we match all the $m_e$ points in $U_e$ with those in $V_e$ as if they were in an isolated line segment. 
If $n_e < m_e$, we select $n_e$ points from $U_e$ that are closer to the middle of the edge and match them with all the points in $V_e$. 
The remaining $m_e - n_e$ unmatched points from $U_e$ 
will seek global matches. 
We denote $U_e^+ \subseteq U_e$ and $V_e^+ \subseteq V_e$, as the remaining point sets on edge $e$ after the above local matching process, respectively. For each edge $e$, exactly one of these two sets will be empty, and the other set will be concentrated near the ends of the edge.

As such, we estimate $\alpha$ by the expected fraction of globally matched points in $U_e$. 
Global matching can happen to a point in $U_e$ only when $m_e>n_e$, and hence $\alpha$ can be estimated by the conditional expectation as the following:
\begin{align}
\label{eq: alpha}
\alpha 
\approx \frac{1}{\mu L} \cdot \text{Pr}\{m_e > n_e\} \cdot \mathbb{E}[m_e-n_e \mid m_e > n_e].
\end{align}
We see that $m_e - n_e$ is the difference between two Poisson random variables, which follows a Skellam distribution and 
can be approximated by a normal distribution with mean $(\mu-\lambda)L$ and variance $(\lambda+\mu)L$.
As such, we have: 
\begin{align}
\label{eq: prob_n_less_m}
&\text{Pr}\{m_e>n_e\} 
\approx \Phi\left(\frac{-\frac{1}{2}+(\mu-\lambda)L}{\sqrt{(\lambda+\mu)L}}\right),
\end{align}
where $\Phi(\cdot)$ is the cumulative distribution function of standard normal distribution.
Then, the conditional expectation is:
\begin{align}
\label{eq: cond_exp_n_less_m}
&\mathbb{E}[m_e-n_e \mid m_e > n_e]
= (\mu-\lambda)L + \sqrt{(\lambda+\mu)L} \cdot \frac{ \phi\left( \frac{-\frac{1}{2}+(\lambda - \mu)L}{\sqrt{(\lambda+\mu)L}} \right)}{1-\Phi \left( \frac{-\frac{1}{2}+(\lambda - \mu)L}{\sqrt{(\lambda+\mu)L}} \right)},
\end{align}
where $\phi(\cdot)$ is the probability density function of the standard normal distribution. Then, $\alpha$ is obtained by plugging Equations \eqref{eq: prob_n_less_m} and \eqref{eq: cond_exp_n_less_m} into Equation \eqref{eq: alpha}. 
Similarly, by symmetry, 
\begin{align}
\label{eq: prob_m_less_n}
&\text{Pr}\{n_e > m_e\} 
\approx \Phi\left(\frac{-\frac{1}{2}+(\lambda - \mu)L}{\sqrt{(\lambda+\mu)L}}\right), \\
\label{eq: cond_exp_n_larger_m}
&\mathbb{E}[n_e-m_e \mid n_e > m_e]
= (\lambda-\mu)L + \sqrt{(\lambda+\mu)L} \cdot \frac{ \phi\left( \frac{-\frac{1}{2}+(\mu-\lambda)L}{\sqrt{(\lambda+\mu)L}} \right)}{1-\Phi \left( \frac{-\frac{1}{2}+(\mu-\lambda)L}{\sqrt{(\lambda+\mu)L}} \right)}.
\end{align}
Now, all that is left is to derive an estimate of $\mathbb{E}[X_{\text{G}}^\text{g}]$. In so doing, for all unmatched point $u\in U_e^+, \forall e\in \mathcal{E}$, we perform the following breadth-first search procedure throughout the network (as illustrated in Figure \ref{fig: nns}) to identify a feasible match globally:  
(i) find the nearer end of edge $e$ from $u$, denoted as $o_0$. Identify the layer of edges, $N_k$, whose nearer end is $kL$ distance away from $o_0$, for all $k = 0,1,\cdots $.
(ii) starting from $k=0$, check whether there exists any edge $e' \in N_k$ such that $V_{e'}^+ \ne \emptyset$. If yes, match $u$ with the nearest point $v \in \bigcup_{e' \in N_k} V_{e'}^+$ (e.g., shown as the labeled red circle in Figure \ref{fig: nns}), mark the edge containing $v$ as $e^*$ 
and the nearer end of $e^*$ (to $o_0$) as $o_{k}$. 
Repeat (i) and (ii) until all points in $U_e^+, \forall e\in \mathcal{E}$ have found a match, or all points in $V_e^+, \forall e\in \mathcal{E}$ have been used for a match.

\begin{figure}[ht!] 
\centering 
\includegraphics[scale=0.75]{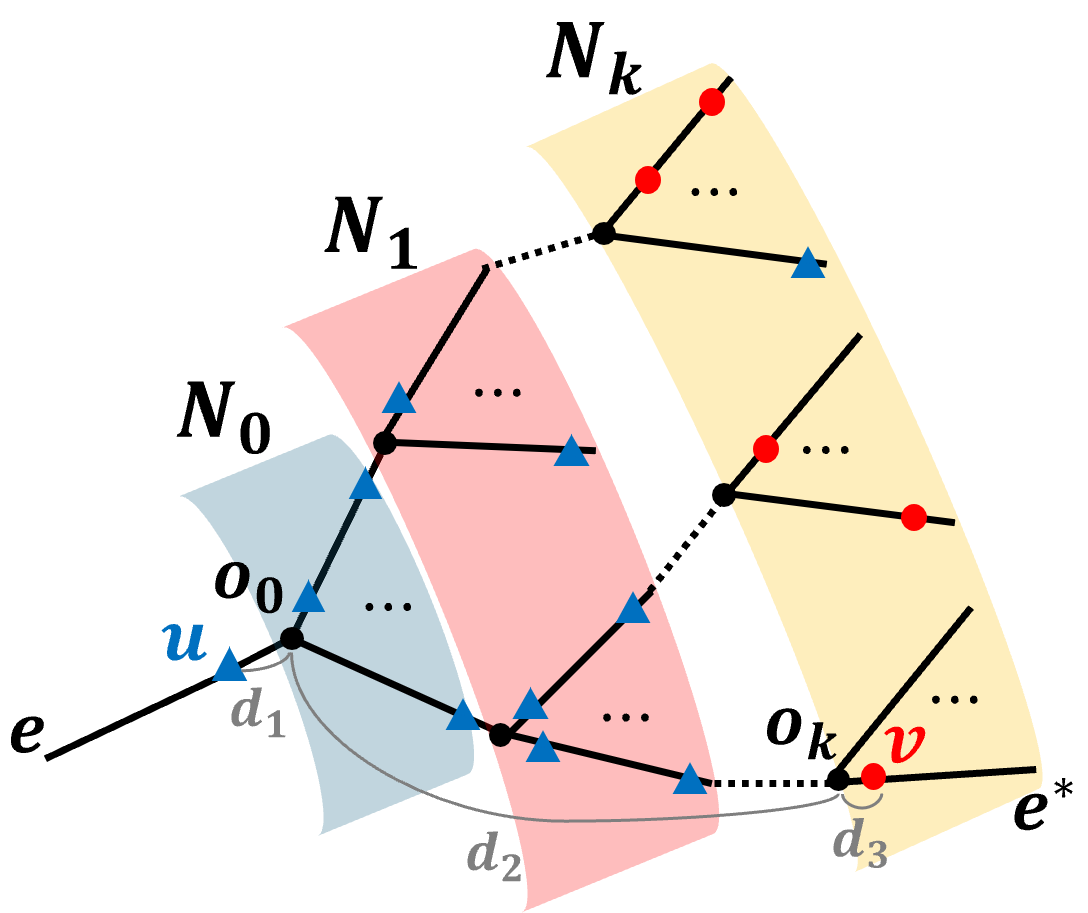}
\caption{Illustration of the breadth-first search procedure.}
\label{fig: nns}
\end{figure}

According to the above matching process, the distance between a specific $u \in U_e^+$ and its match $v \in V_{e^*}^+$ consists of three parts (as indicated by the three gray curves in Figure \ref{fig: nns}): (i) the distance from $u$ to $o_0$ on edge $e$; (ii) the distance from $o_0$ to $o_{k}$, where $e^* \in N_k$; (iii) the distance from $o_{k}$ to $v$ on edge $e^*$. 
Let random variables $d_1$, $d_2$ and $d_3$ represent these three distances, and $\mathbb{E}[X_{\text{G}}^\text{g}]$ can be written as the sum of their individual expectations; i.e., 
\begin{align}
\label{eq: exp_XG_d1d2d3}
\mathbb{E}[X_{\text{G}}^\text{g}] &= \mathbb{E}[d_1]+ \mathbb{E}[d_2] + \mathbb{E}[d_3]. 
\end{align}

First, we look at 
$d_2$. 
The probability of finding a match in the $k$-th layer should be no larger than the probability for at least one edge $e' \in N_k$ to have $n_{e'} > m_{e'}$. 
Since all edges are translationally symmetric, $\text{Pr}\{n_{e'} > m_{e'}\} = \text{Pr}\{n_{e} > m_{e}\}$. 
The probability of finding a match in the $k$-th layer is:
\begin{align}
1-(1- \text{Pr}\{n_{e} > m_{e}\})^{|N_k|},\quad k = 0, 1, \cdots.
\end{align}
We would have $d_2 = kL$ (for $k = 0,1, 2, \cdots$) if a match is successfully found in the $k$-th layer, but not in any previous layers; hence, for $k = 0, 1, 2, \cdots$: 
\begin{align}
\label{eq: pr_d2_kL}
&
\text{Pr} \{d_2 = kL\}
\approx 
(1- \text{Pr}\{n_e > m_e\})^{\sum_{i=0}^{k-1}|N_i|}\cdot [1-(1- \text{Pr}\{n_e > m_e\})^{|N_k|}].
\end{align}

Note that the definition of edge layers $\{N_k\}$ directly captures the presence and density of edge cycles in the network, and the
exact cardinality of each layer $k$, $|N_k|$, can be easily counted for any given 
topology. 
\footnote{
One can also estimate $|N_k|$ recursively across successive layers using the edge-expansion ratio of a graph, which measures the relative number of edges crossing a graph partition between two subsets of vertices as compared to the number of vertices in the smaller partition. 
\citet{ellis2011expansion} shows that a random $D$-regular graph's edge-expansion ratio almost surely approaches a lower bound of $\frac{D}{2}$ as $D\to\infty$.
}
Given each node has an equal degree $D$ in the current network, we have $|N_k| \le (D-1)^{k+1}, \forall k$, in general, and $|N_k| \approx (D-1)^{k+1}$ when $k$ is relatively small. 
Since the first term in Equation \eqref{eq: pr_d2_kL} quickly approaches zero as $k$ increases, while the second term approaches a constant, only the first few searches with small $k$ matters. As such, $|N_k| = (D-1)^{k+1}$ can serve as a good approximation in practice. A comparison between the exact and approximated values for $|N_k|$ is provided in the numerical experiment section.
$\mathbb{E}[d_2]$ can be approximated by the following truncation: 
\begin{align}
&\mathbb{E}[d_2] \approx \sum_{k=0}^{\kappa}kL \cdot \text{Pr} \{d_2 = kL\},
\end{align} 
where $\kappa$ is a relative small value (e.g., about 10).

Next, we look at $d_1$ and $d_3$. 
Recall that all unmatched points in $U_e^+$, if any, after local matching, will be located near the two ends of edge $e$. 
The expected number of unmatched points near each end is $\frac{\mathbb{E}[|U_e^+|]}{2}=\frac{1}{2}\mathbb{E}[m_e-n_e \mid m_e > n_e]$. The average distance between two adjacent points among them is $\frac{1}{\mu}$.
Then, the distance between an arbitrary point $u \in U_e^+$ and its nearer end $o_0$, i.e., $d_1$, is approximately uniformly distributed in the interval $\left( 0, \frac{1}{2\mu}\mathbb{E}[m_e-n_e \mid m_e > n_e] \right)$, and hence the average, $\mathbb{E}[d_1]$, is approximately equal to half of the interval length, i.e.,
\begin{align}
\mathbb{E}[d_1] 
\approx \frac{1}{4\mu} \cdot \mathbb{E}[m_e-n_e \mid m_e > n_e],
\end{align}
where $\mathbb{E}[m_e-n_e \mid m_e > n_e]$ is given by Equation \eqref{eq: cond_exp_n_less_m}.
The derivation of $\mathbb{E}[d_3]$ is similar, but through symmetrical analysis of the unmatched points in $V_{e^*}^+$ where $e^* \in N_k$. The distance between $o_k$ and an arbitrary unmatched point in $V_{e^*}^+$ near $o_k$ is approximately uniformly distributed in the interval $\left( 0, \frac{1}{2\lambda} \mathbb{E}[n_{e^*}-m_{e^*} \mid n_{e^*} > m_{e^*}]\right)$, and again, $\mathbb{E}[n_{e^*}-m_{e^*} \mid n_{e^*} > m_{e^*}] = \mathbb{E}[n_e-m_e \mid n_e > m_e]$. However, competition may occur, and not all points in $V_{e^*}^+$ must be matched. 
From the perspective of a specific point $u\in U_e^+$, during the breadth-first search process, it will be matched with the nearest available point $v \in \bigcup_{e' \in N_k} V_{e'}^+$. 
However, 
other competing points in $\bigcup_{e\in \mathcal{E}} U_e^+$
(as indicated by the other blue triangles in Figure \ref{fig: nns}) may also have the chance to be matched with the points in $\bigcup_{e' \in N_k} V_{e'}^+$ (shown as the red circles in Figure \ref{fig: nns}). 
The expected ratio between the total number of competing points of $u$ and the total number of available points in $\bigcup_{e' \in N_k} V_{e'}^+$ is approximately $\frac{|U|}{|V|}=\frac{\mu}{\lambda}$. 
This indicates that at the end of the breadth-first search process, $\frac{\mu}{\lambda}$ of the points in $\bigcup_{e' \in N_k} V_{e'}^+$ will be matched. 
Since $e^* \in N_k$, the distance between $o_k$ and an arbitrary matched point $v \in V_{e^*}^+$ near $o_k$, i.e., $d_3$, is approximately uniformly distributed in the interval $\left( 0, \frac{\mu}{2\lambda^2} \mathbb{E}[n_{e}-m_{e} \mid n_{e} > m_{e}]\right)$, and hence the average, $\mathbb{E}[d_3]$, can be approximately estimated as follows: 
\begin{align}
\label{eq: d3}
\mathbb{E}[d_3] 
\approx \frac{\mu}{4\lambda^2} \cdot \mathbb{E}[n_e - m_e \mid n_e>m_e],
\end{align}
where $\mathbb{E}[n_e-m_e \mid n_e > m_e]$ is given by Equation \eqref{eq: cond_exp_n_larger_m}. 

Summarizing all the above, $\mathbb{E}[X_{\text{G}}]$ can be estimated out of Equations \eqref{eq: X_G}-\eqref{eq: d3}.

\section{Numerical Experiment}
\label{sec: numerical}

\subsection{Verification of 1D RBMP on A Lattice}

In this section, we validate the accuracy of the proposed formulas for 1D RBMP on a lattice using a series of Monte-Carlo simulations.
For each combination of $n$ and $m$ values, 500 RBMP realizations are randomly generated. Each realized instance is solved by a customized algorithm, as detailed in \ref{apx: DP}.
The average optimal matching distances for each $(m,n)$ combination is recorded as the sample mean across the 500 realizations. 

Figure \ref{fig: numerical} compares the simulation results with the formulas developed for both balanced and unbalanced cases, including Equations \eqref{eq: exp_Xnn}, \eqref{eq: exp_Xnm}, and \eqref{eq: exp_Xnm_upper}. 
The distance estimates are rescaled by an asymptotic scaling factor --- $\sqrt{n}$ and $n+m$ for balanced and unbalanced cases, respectively. 
The sample means of optimal matching distances solved for each instance from the Monte-Carlo simulation is represented by the red solid curve with square markers. 
The estimations from Equations \eqref{eq: exp_Xnn}, \eqref{eq: exp_Xnm}, \eqref{eq: exp_Xnm_upper} are marked by the blue dashed curves, the blue dashed curves with cross markers, and the green dot-dash curves with plus markers, respectively.  


\begin{figure}[ht!] 
\centering 
\includegraphics[scale=0.65]{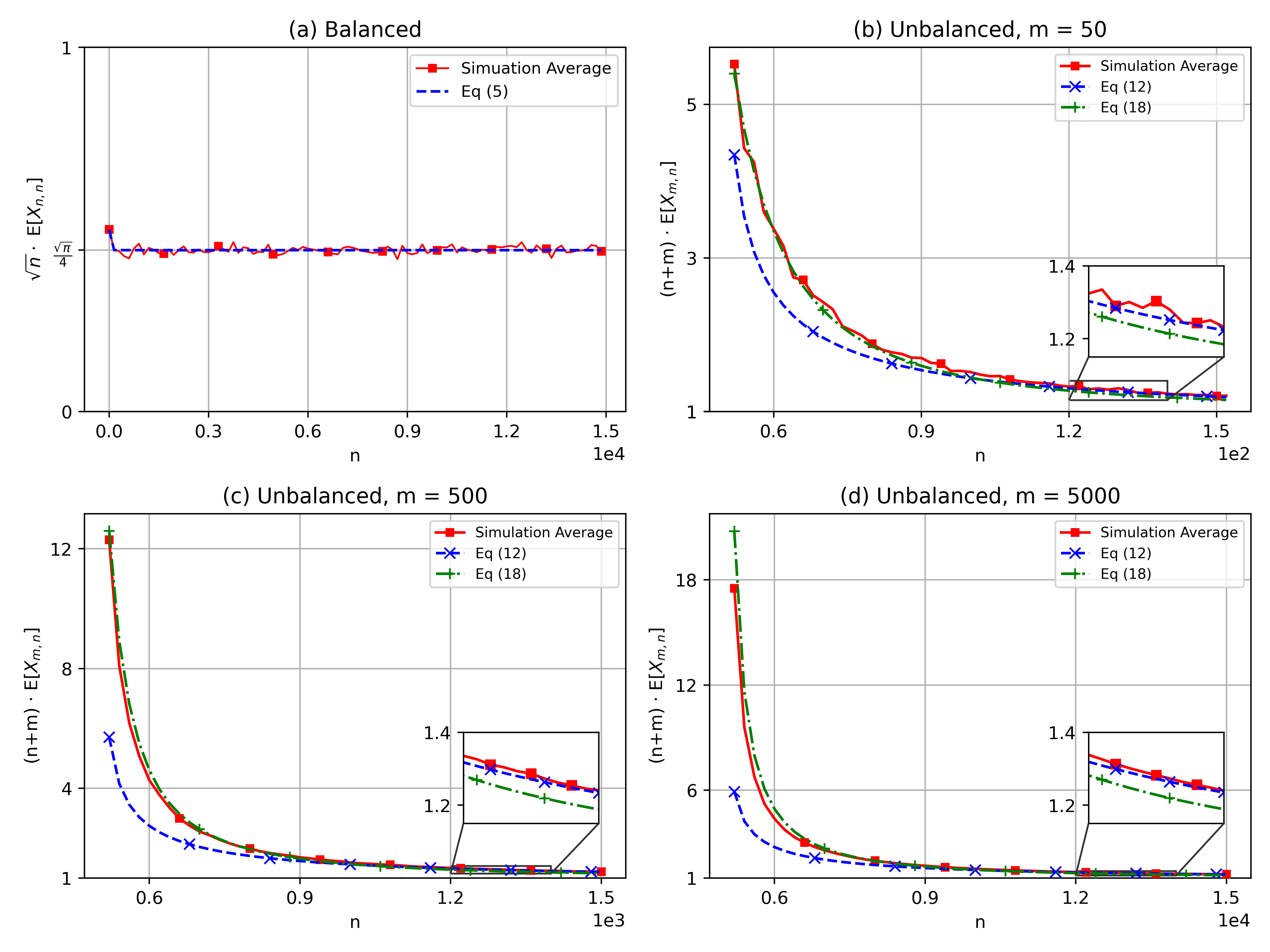}
\caption{Verification of 1D RBMP on a lattice.}
\label{fig: numerical}
\end{figure}

For the balanced cases, we let the value of $n = m$ vary from 1 to 15000. 
Figure \ref{fig: numerical} (a) shows the results. 
It can be seen that the estimations by Equation \eqref{eq: exp_Xnn} closely match with the simulation averages, 
with an average relative error of 1.62\%. 
This indicates that Equation \eqref{eq: exp_Xnn} provides very accurate estimates. 
It can also be observed that the simulated distance quickly converges to $\frac{\sqrt{\pi}}{4}$ as $n$ increases, which is consistent with our analytical predictions. 

Next, for the unbalanced cases, we set $m\in\{50, 500, 5000\}$, and let $n$ range from $m+1$ to a sufficiently large number $3m$. 
Figures \ref{fig: numerical} (b)-(d) show the results. 
It can be first observed that the estimations from Equation \eqref{eq: exp_Xnm_upper} closely match with the simulation averages across all $n$ and $m$ values. The average relative errors are 3.83\%, 3.89\%, and 5.18\% for $m=50, 500, 5000$, respectively. This indicates that, in general, Equation \eqref{eq: exp_Xnm_upper} can provide very accurate distance predictions. 
We then look at the estimations from Equations \eqref{eq: exp_Xnm}. 
It is clear that, when $n\gg m$, the equation also matches quite well with the simulation averages. 
Specifically, for $n \geq 2m$, the average relative errors for Equation \eqref{eq: exp_Xnm} are 3.17\%, 2.92\%, and 2.86\%, for $m=50, 500, 5000$, respectively. 
When $n\rightarrow m$, larger discrepancies can be observed. 
Recall from Section \ref{subsec: stars_and_bars}, this discrepancy may arise from the i.i.d assumption for point selections in $V^*$. 
Observations from various $V^*$ instances show that, when $n \approx m$, points at certain specific positions (e.g., the first or the last point when $n = m + 1$) are more likely to be selected in $V^*$ than the others. 
Nevertheless, Equation \eqref{eq: exp_Xnm} provides very good estimates when $n \ge 2m$. 
In addition, the distances in all cases converge to 1 as $n \rightarrow \infty$, as expected.

In summary, one can choose the most suitable formula given the specific problem setup and the required accuracy. For balanced cases, Equation \eqref{eq: exp_Xnn} should be used. For unbalanced cases, when $n \gg m$, Equation \eqref{eq: exp_Xnm} is recommended as it can already provide a good estimation and is computationally more efficient than Equation \eqref{eq: exp_Xnm_upper}; otherwise, Equation \eqref{eq: exp_Xnm_upper} is more suitable as it provides the most accurate estimates.

\subsection{Verification of 1D RBMP Variants}
In this section, we validate the accuracy of the proposed formulas for three 1D RBMP variants.
For any parameter combination, 500 RBMP instances are randomly generated from Monte-Carlo simulations. Instances with periodic boundaries or uniform point distribution are solved through customized algorithms described in \ref{apx: DP}. Instances with arbitrary-length line are solved through Equation \eqref{eq: def}-\eqref{eq: def_ctd} by a standard linear program solver Gurobi.
The optimal matching distances are then averaged across the 500 realizations.

We first try balanced cases with periodic boundaries. Let the value of $n = m$ vary from 1 to 15000. Figure \ref{fig: variant}
(a) compares the simulated average distances, represented by the red solid curves with square markers, with predictions from Equation \eqref{eq: exp_Xnn_periodic}, represented by the blue dashed curves.
All estimates are rescaled by a factor of $\sqrt{n}$. 
It can be seen that the predictions by Equation \eqref{eq: exp_Xnn_periodic} closely match with the simulation averages. 
When $n$ and $m$ are considerably small, Equation \eqref{eq: exp_Xnn_periodic} tends to underestimate the simulation average. However, this error diminishes rapidly as $n$ increases. For instance, the relative error is 29.2\% when $n=1$, but drops significantly to 2.80\% when $n=3$. 
When $n$ is sufficiently large (e.g., $n > 10$), Equation \eqref{eq: exp_Xnn_periodic} can provide a very accurate prediction. 

Next, we study cases with uniformly distributed points. We set $m=5000$, and let $n$ range from $m+1$ to a sufficiently large number $3m$. 
Simulation averages and formula predictions are rescaled by a factor of $n+m$. 
Figure \ref{fig: variant} (b) shows the results. 
Similar to the cases on a lattice, the predictions from Equation \eqref{eq: exp_Xnm_upper_corrected} (the green dot-dash curves) closely match with the simulation averages (the red solid curve with square markers), with an average relative error 8.01\%. 
This indicates that, in general, Equation \eqref{eq: exp_Xnm_upper_corrected} can provide very accurate distance predictions. 
In contrast, Equation \eqref{eq: exp_Xnm_corrected} (the blue dashed curve) matches quite well with the simulation averages when $n\gg m$, with the average relative error 6.57\%. Yet, it yields a larger error for small $n$ values 
due to the i.i.d assumption for point selections in $V^*$ as discussed for Equation \eqref{eq: exp_Xnm}. 

Last, for the cases in an arbitrary-length line, we fix $\mu$ but vary $L$ and $\lambda$. Figure \ref{fig: variant} (c) compares the simulated average distances, 
represented by markers, with predictions from Equations \eqref{eq: exp_Xnn_scale} (for $\lambda/\mu=1$) and \eqref{eq: exp_XE_upper_corrected} (for $\lambda/\mu \in \{1.1, 1.5, 3\}$), represented by lines. The line length $L$ varies from $\{1,3,5,7,9\}$ [du], and $\mu = 10$ [1/du].
It can be observed that, in general, the estimations by Equations \eqref{eq: exp_Xnn_scale} and \eqref{eq: exp_XE_upper_corrected} fit tightly with the simulation averages across all $L$ and $\lambda/\mu$ values. The average relative errors are 2.50\%, 3.22\%, 1.51\%, and 8.12\% for $\lambda/\mu = 1,1.1,1.5, 3$, respectively. Also, it is clear that 
the optimal matching distance increases concavely with $L$ when $\lambda/\mu = 1$. Yet, as the ratio $\lambda/\mu$ becomes slightly larger, both the simulation averages and the formula estimations become flatter. 
At higher values of $\lambda/\mu$, such as $1.5$ or $3$, the optimal matching distance barely changes with $L$. 
These observations support the discussion in Section \ref{subsec: scaling}, and visually illustrates how the average optimal matching distance scales with $\sqrt{L}$ in the balanced RBMP, but is largely independent of $L$ in the unbalanced RBMP with $\lambda \gg \mu$. 

\begin{figure}[ht!] 
\centering 
\includegraphics[width = 1\textwidth]{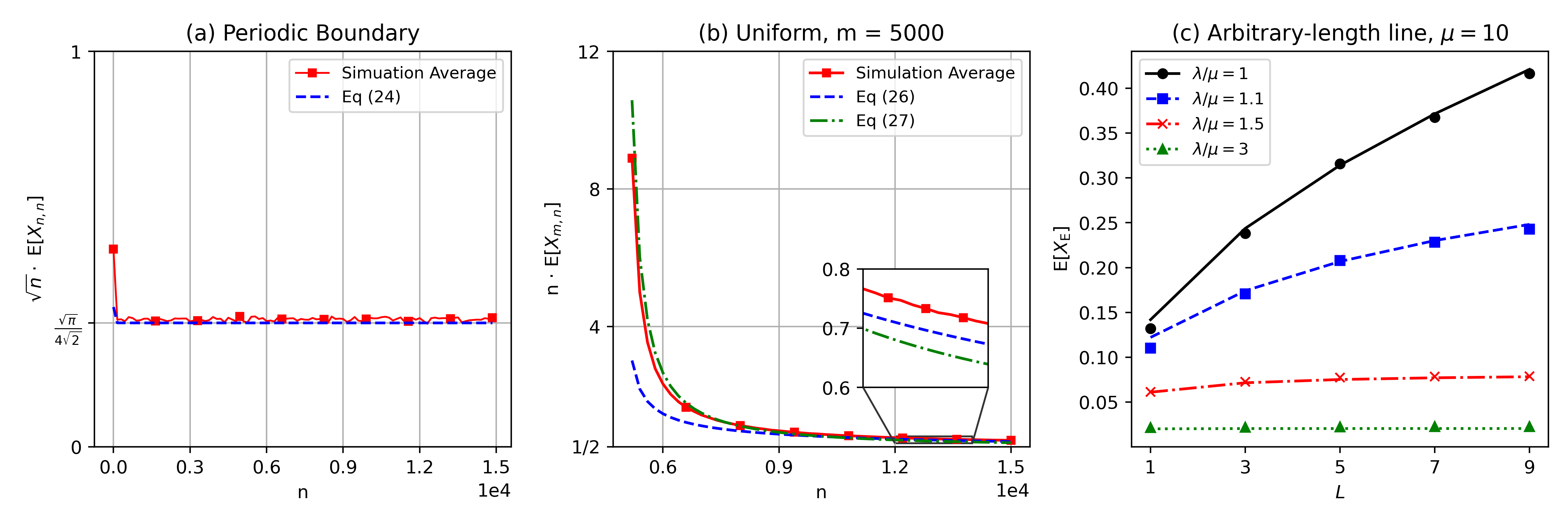}
\caption{Verification of 1D RBMP variants.}
\label{fig: variant}
\end{figure}

{
\color{Black}
\subsection{Time Complexity Comparison for 1D RBMP Estimators}
While the proposed formulas have been shown to be accurate, it is instructive to compare their computational performance against that of conventional simulation-based numerical estimates. 
In this section, we compare the time complexity and accuracy of computing the proposed formulas with the 
estimates from Monte-Carlo simulations across different sample sizes. 
We report findings on 1D RBMP on a lattice with open boundaries, which can well represent the computational performance in other problem settings. 

Recall that the three formulas for 1D RBMP on a lattice have the following computational time complexities: 
(i) Equation \eqref{eq: exp_Xnn}, which applies to the balanced case ($m = n$), has a constant time complexity $\mathcal{O}(1)$; (ii) Equation \eqref{eq: exp_Xnm}, which applies when the supply is much greater than the demand ($n \gg m$), has a linear time complexity $\mathcal{O}(m)$; and (iii) Equation \eqref{eq: exp_Xnm_upper}, which applies to the general unbalanced case ($m \neq n$), has a polynomial time complexity: $\mathcal{O}((n-m)m^2)$ for all $m \times n$ cases. 
In contrast, 
for simulation-based estimation, a standard algorithm for solving one single matching problem instance has a worst-case time complexity of $\mathcal{O}(n^3)$. 
As such, if we only make the comparison in terms of these worst-case time complexities, solving one simulated problem instance already incurs a higher or equal complexity as compared to any of the proposed formulas. Solving a larger sample of instances (to obtain reliable numerical estimates) will obviously lead to an even greater worst-case computational burden.

In terms of average complexity, we numerically compare the actual (average) running times of the formulas under different parameter combinations $(m, n)$ with those of simulation-based estimations across different sample sizes, denoted by $N$. 
With proper treatment for a given problem instance, such as an informed initialization or cost scaling, 
the average running time to solve one instance can be lower. 
In so doing, we apply a state-of-art method, the modified Jonker–Volgenant algorithm \citep{Scipy}, to solve each simulated instance. 

Here, as an illustration, 
we examine four parameter combinations with a fixed $m = 50$ and varying $n \in \{ m, 1.5m, 2m, 6m\}$, representing the balanced, nearly balanced, unbalanced, and highly unbalanced cases, respectively. 
We also vary $N \in \{10, 10^2, 10^3, 10^4, 10^5\}$. 
For each $(n,N)$ combination, we conduct 100 simulation runs. In each run, we generate $N$ RBMP instances, solve the matchings, and record the average optimal matching distance. 
We also record the mean absolute percentage error (MAPE) and the algorithm running time for each simulation run, and report the sample mean of these values over 100 runs for each $(n,N)$ combination. 
The ``ground-truth" for a given $n$ is assumed to be the sample mean over $10^6$ problem instances.

\begin{figure}
    \centering
    \includegraphics[width=1\linewidth]{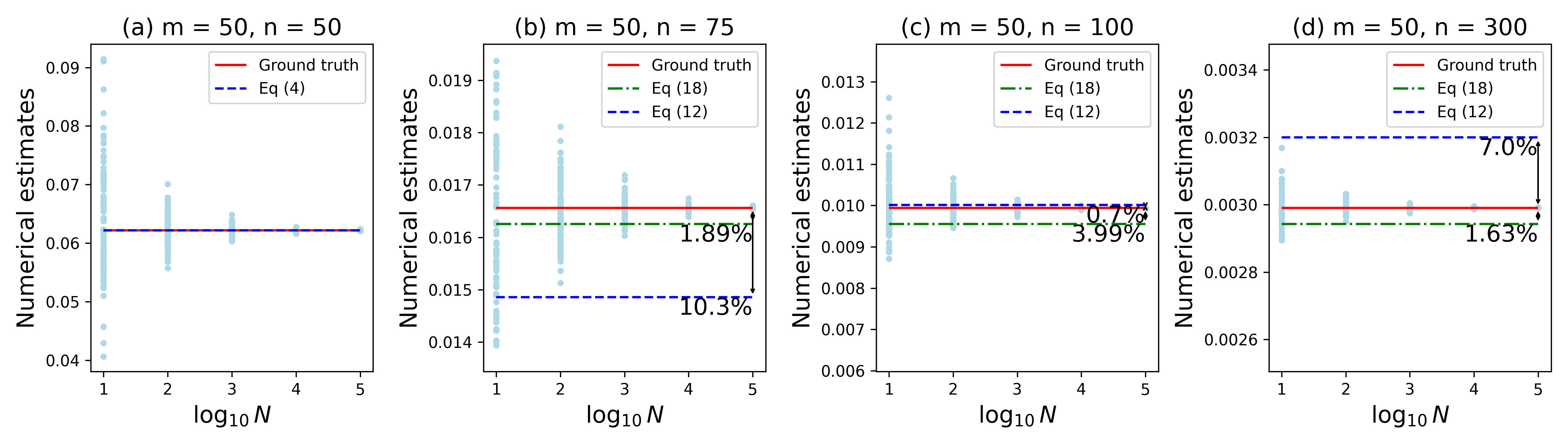}
    \caption{Convergence and accuracy comparison between numerical estimates and formulas.}
    \label{fig:convergence}
\end{figure}

Figures \ref{fig:convergence} (a)-(d) illustrate the convergence pattern of simulation-based estimators across different sample sizes. The average optimal matching distance obtained from each $N$-sized sample is represented by a light-blue dot. The corresponding estimates from Equations \eqref{eq: exp_Xnn}, \eqref{eq: exp_Xnm} and \eqref{eq: exp_Xnm_upper} are represented by the blue dash and green dash-dot curves, respectively. The ground-truth is represented by the red solid line. 
It can be seen that when $N$ is smaller, the simulation results (the blue dots) show clear scatters; 
a sufficiently large sample size is required for the simulation-based estimator to achieve comparable accuracy to that of the analytical formulas. 

Table \ref{tab:convergence} provides a summary of MAPEs and computation times of Equations \eqref{eq: exp_Xnn}, \eqref{eq: exp_Xnm}, \eqref{eq: exp_Xnm_upper}, and the simulation estimators. 
For the balanced case, Equation \eqref{eq: exp_Xnn} predicts with MAPE $0.0130\%$ and running time $4.29\times 10^{-5}$ seconds, which outperforms the simulation estimator with all considered sample sizes (i.e., the minimum MAPE is $0.116\%$ at $N = 10^5$ and the shortest running time is $8.76\times 10^{-4}$ seconds at $N = 10$).  
For the unbalanced cases, an appropriate sample size must be considered to fairly compare the simulation estimators and the formulas. 
For example, it is observed that, to reach the same level of MAPE as Equation \eqref{eq: exp_Xnm_upper}, approximately $10^3, 10^2, 10^1$ samples are needed for $n = 75, 100, 300$, respectively. 
The corresponding running times required by the simulation-based estimators are $4.75\times 10^{-2}, 4.93\times 10^{-3}, 6.84\times 10^{-4}$ seconds. 

Also, it is not surprising that, as $n$ increases, a smaller sample size becomes sufficient for the simulation-based estimator. This can be explained by the asymptotic behavior of $\mathbb{E}[X_{m,n}]$: as $n \to \infty$, $\mathbb{E}[X_{m,n}]\to l$, and the variance of $X_{m,n}$ converges to zero. A more detailed discussion can be found in a working paper \citep{shen_zhai_ouyang_2024}. 

Nevertheless, our proposed formulas still have strong advantages over simulation-based estimators. 
Specifically, Equation \eqref{eq: exp_Xnm} is significantly faster than the simulation-based estimator across all cases due to its linear time complexity. 
In addition, Equation \eqref{eq: exp_Xnm_upper} has the advantage that the recursive computation can cache all the intermediate results. 
As in the current example, an all-inclusive recursive computation of $\mathbb{E}[X_{50, 300}]$ yields the values of $\mathbb{E}[X_{m, n}]$ for all combinations of $1 \le m\leq 50, 1 \le n\leq 300$ (i.e., a total of $15,000$ values). Although this all-inclusive cost is relatively larger in the highly unbalanced case (i.e., $n=300$), the average computation time per parameter combination (around $3\times10^{-7}$ seconds) is quite attractive. }

\begin{table}[htbp]
\centering
\caption{Convergence and accuracy comparison between analytical formulas and simulation estimators.}
\label{tab:convergence}
\rotatebox{90}{
\begin{tabular}{|c|c|c|c|c|c|c|c|c|c|c}
\hline
\multirow{2}{*}{$n$} & \multirow{2}{*}{$N$}
  & \multicolumn{4}{c|}{MAPE}
  & \multicolumn{4}{c|}{Total running time (sec)} \\ \cline{3-10}
 & & Eq.~(5) & Eq.~(12) & Eq.~(18) & Sample mean
   & Eq.~(5) & Eq.~(12) & Eq.~(18) & Sample mean \\ \hline

\multirow{5}{*}{50}
 & $10^1$ & \multirow{5}{*}{0.0130\%} & \multirow{5}{*}{/} & \multirow{5}{*}{/} & 12.1\% & \multirow{5}{*}{$4.29 \times 10^{-5}$}&\multirow{5}{*}{/} & \multirow{5}{*}{/} & 8.76$\times 10^{-4}$\\ 
\cline{2-2}\cline{6-6}\cline{10-10}
 & $10^2$ & & &  & 3.45\% &  &  & & $8.30\times 10^{-3}$ \\ 
\cline{2-2}\cline{6-6}\cline{10-10}
 & $10^3$ & & &  & 1.22\% &  & & & $8.32\times 10^{-2}$ \\ 
\cline{2-2}\cline{6-6}\cline{10-10}
 & $10^4$ & & &  & 0.307\% &  & & & $0.871$ \\ 
\cline{2-2}\cline{6-6}\cline{10-10}
 & $10^5$ & & &  & 0.116\% &  & & & 8.63 \\
\cline{1-7}\cline{8-10}

\multirow{5}{*}{75}
 & $10^1$ & \multirow{5}{*}{/} & \multirow{5}{*}{10.3\%} & \multirow{5}{*}{1.89\%} & 8.84\% &\multirow{5}{*}{/} &\multirow{5}{*}{$9.21\times 10^{-5}$} & \multirow{5}{*}{\shortstack{$1.08\times 10^{-3}$\\ for 3,750 values\\
 (avg. $2.88\times 10^{-7}$)}} & $4.03\times 10^{-4}$\\ 
\cline{2-2}\cline{6-6}\cline{10-10}
 & $10^2$ & & &  & 2.89\% &  &  & & $4.14\times 10^{-3}$ \\ 
\cline{2-2}\cline{6-6}\cline{10-10}
 & $10^3$ & & &  & 0.976\% &  & & & $4.75\times 10^{-2}$ \\ 
\cline{2-2}\cline{6-6}\cline{10-10}
 & $10^4$ & & &  & 0.256\% &  & & & $0.428$ \\ 
\cline{2-2}\cline{6-6}\cline{10-10}
 & $10^5$ & & &  & 0.100\% &  & & & $4.21$ \\
\cline{1-7}\cline{8-10}

\multirow{5}{*}{100}
 & $10^1$ & \multirow{5}{*}{/} & \multirow{5}{*}{0.7\%} & \multirow{5}{*}{3.99\%} & 5.90\% &\multirow{5}{*}{/} & \multirow{5}{*}{$5.82\times 10^{-5}$} & \multirow{5}{*}{\shortstack{$1.79\times 10^{-3}$\\ for 5,000 values \\
 (avg. $3.58\times 10^{-7}$)}} & $4.41\times10^{-4}$ \\
 
\cline{2-2}\cline{6-6}\cline{10-10}
 & $10^2$ & & &  & 2.00\% &  &  & & $4.93\times 10^{-3}$ \\ 
\cline{2-2}\cline{6-6}\cline{10-10}
& $10^3$ & & &  & 0.560\% &  & &  & $4.87\times 10^{-2}$ \\ 
\cline{2-2}\cline{6-6}\cline{10-10} 
& $10^4$ & & &  & 0.185\% &  & & & $0.494$ \\ 
\cline{2-2}\cline{6-6}\cline{10-10}
 & $10^5$ & & &  & 0.0599\% &  & & & 4.91 \\ 
\cline{1-7}\cline{8-10}

\multirow{5}{*}{300}
 & $10^1$ & \multirow{5}{*}{/} & \multirow{5}{*}{7.0\%} & \multirow{5}{*}{1.63\%} & 1.27\% & \multirow{5}{*}{/} & \multirow{5}{*}{$5.34\times10^{-5}$} &
 \multirow{5}{*}{\shortstack{$4.82\times 10^{-3}$\\ for 15,000 values \\ (avg. $3.21\times 10^{-7}$)}} 
 & $6.84\times 10^{-4}$ \\ 
\cline{2-2}\cline{6-6}\cline{10-10}
 & $10^2$ &  & & & 0.440\% &  & & & $7.70\times 10^{-3}$ \\ 
\cline{2-2}\cline{6-6}\cline{10-10}
 & $10^3$ &  & & &  0.139\%&  & & & $7.47\times 10^{-2}$ \\ 
\cline{2-2}\cline{6-6}\cline{10-10}
 & $10^4$ &  & & & 0.0430\% &  & & & $0.752$ \\ 
\cline{2-2}\cline{6-6}\cline{10-10}
 & $10^5$ &  & & & 0.0151\% &  & & &  7.52\\ \hline
\end{tabular}
}
\end{table}


\subsection{Verification of Network RBMP}

In this section, we validate the accuracy of the proposed formulas for 
We fix $L=1$ du and $\mu = 5$ [1/du], but vary $D$ and $\lambda$. For any parameter combination, 100 RBMP instances are randomly generated, each solved through Equation \eqref{eq: def}-\eqref{eq: def_ctd} by a standard linear program solver, and the optimal matching distances are averaged across the 100 realizations.

We build a series of $D$-regular networks with node degree $D \in \{3,4,6\}$, each with 36 total number of unit-length edges. 
We further vary $\lambda$ from 5 to 25 [1/du]. 
Figures \ref{fig: network} (a)-(c) compare the simulation averages (red solid curve) with estimations by Equation \eqref{eq: exp_XE_upper_corrected} (blue dashed curves), Equation \eqref{eq: X_G} with exact $|N_k|$ (green dot-dash curves), and Equation \eqref{eq: X_G} with approximated $|N_k|= (D-1)^{k+1}$ (green cross markers). 
Each Monte-Carlo simulation instance, represented by a light-blue dot, is also plotted. 
It is observed that the estimations by Equation \eqref{eq: X_G} fit tightly to the simulation average across all parameter combinations, regardless of whether $|N_k|$ is approximated or not. The average relative errors of Equation \eqref{eq: X_G} with exact $|N_k|$ values are 8.45\%, 4.73\%, and 3.40\% for $D = 3, 4, 6$, respectively. With approximated $|N_k|$ values, the average relative errors are 8.54\%, 4.74\%, and 3.40\% for $D = 3, 4, 6$, respectively. 
The difference between the two estimations is almost negligible. 
As previously explained, this is primarily because the probability of matching quickly approaches zero as $k$ increases, and the existence of loops in the network barely has any effect on the results.

This indicates that, in general, Equation \eqref{eq: X_G} with approximated $|N_k|$ can already provide very accurate predictions in a wide range of $D, L, \lambda/\mu$ combinations.  
In the meantime, we observe that Equation \eqref{eq: exp_XE_upper_corrected} also estimates the average distance accurately when $\lambda\gg \mu$. 
Specifically, when $\lambda \geq 2\mu$, the average relative errors of Equation \eqref{eq: exp_XE_upper_corrected} are 9.31\%, 6.72\%, and 5.96\% for $D=3,4,6$, respectively. Recall from Section \ref{sec: network model}, as $\lambda\gg \mu$, there are more chances for a point in $U$ to be matched locally, which indicates that $\alpha$ converges to 0 and Equation \eqref{eq: exp_XE_upper_corrected} will predict the distance almost as as well as Equation \eqref{eq: X_G}.
\begin{figure}
    \centering
    \includegraphics[width=1\textwidth]{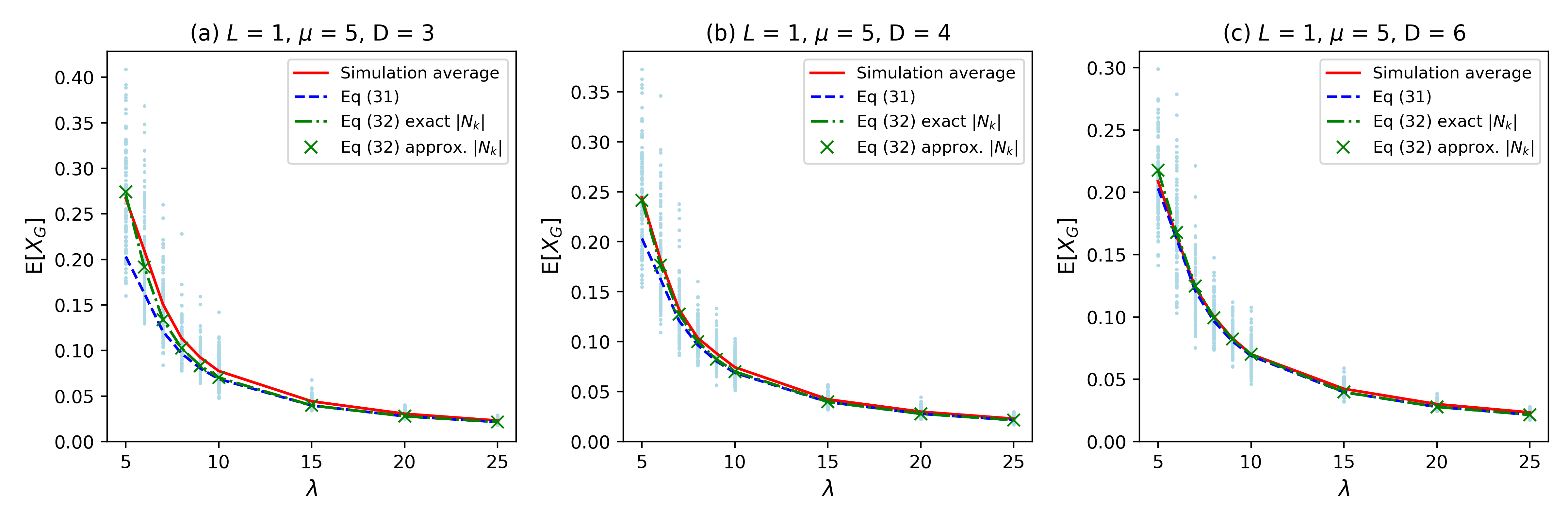}
    \caption{Verification of network RBMP}
    \label{fig: network}
\end{figure}

\section{Conclusion}
\label{sec: Conclusion}
This paper presents a set of closed-form formulas, without curve-fitting or statistical parameter estimation, that can provide accurate estimates for random bipartite matching problems in one-dimensional spaces and networks. These formulas can be directly used in mathematical programs to evaluate and plan resources for many transportation services.
We first study the one-dimensional RBMP on a lattice and
relate the matching distance to the area between a random walk path and the x-axis, and then derive a closed-form formula for balanced matching. 
For unbalanced matching, we first develop a closed-form but approximate formula by analyzing the properties of unbalanced random walks following the optimal removal of a subset of unmatched points. Then, we introduce a set of recursive formulas that yields tight upper bounds based on the analysis of unbalanced random walks. 
Next, three one-dimensional problem variants are discussed, including RBMPs with periodic boundaries, continuously distributed points, and arbitrary-length lines.
Building upon these results, we derive the expected optimal matching distance in a regular graph, in which points are distributed on equal-length edges based on spatial Poisson processes, by quantifying the expected distance when the point is locally matched (i.e., matched within the same edge) or globally matched (i.e., matched across different edges). 
Results indicate that our proposed formulas all provide quite accurate distance estimations for one-dimensional line segments and networks under different conditions. 

Nevertheless, our proposed models build upon several assumptions and approximations, which may be relaxed in the future. 
\textcolor{Black}{
For example, this paper assumes that sample points are distributed either on a lattice or uniformly, the matching cost between points is linear, and the boundary is either open or periodic -- all these assumptions are motivated by transportation applications. It will be interesting to build upon ideas from this paper to explore alternative problem settings that may be suitable for other application contexts. In addition, 
}%
\textcolor{Black}{
in Section \ref{subsec: correction}, we used results from the lattice problem as first-order approximations to the unbalanced match distances between uniformly distributed points.
This approach leads to an overestimation, as shown in Proposition \ref{prop: uniform_vs_lattice}, particularly when $n \gg m$.} Although we propose an approximate correction term in Section \ref{subsec: correction} to address this issue, alternative (better) models could be explored in the future. Further analysis could also be conducted to understand how such correlations among matched pairs would affect the optimal point removals for unbalanced problems. 
In addition, while Section \ref{sec: numerical} shows that the distance estimator by Equation \eqref{eq: exp_Xnm_upper} is quite accurate, it requires solving a recursive formula, which is computationally more cumbersome. It would be interesting to explore alternative methods that can provide simpler estimates of similar accuracy. 
\textcolor{Black}{
Moreover, \ref{apx: DP} shows that the proposed DP method provides a fast way to compute one-dimensional bipartite matching instances. It will be worthwhile to conduct a precise time complexity analysis and compare its performance against that of alternative methods (e.g., the cost scaling algorithm \citep{GOLDBERG19971} and the modified Jonker-Volgenant algorithm \citep{Scipy}).} For networks, this paper opens the door to many interesting new questions. 
For example, we plan to take a deeper dive into the properties of balanced optimal matching (where matching pairs are more strongly correlated) in 
certain types of graphs, such as trees, cactus graphs, and series-parallel graphs. Moreover, future research should develop methods to estimate the expected matching distance in networks with varying edge lengths, varying degrees, and heterogeneous point densities.

\section*{Acknowledgments}
The very helpful comments from the editors and reviewers are gratefully acknowledged.

\bibliographystyle{apalike}  
\bibliography{Batch_1d}  

\begin{thebibliography}{}

\bibitem[Abeywickrama et~al., 2022]{Abeywickrama_network_2022}
Abeywickrama, T., Liang, V., and Tan, K.-L. (2022).
\newblock Bipartite matching: What to do in the real world when computing assignment costs dominates finding the optimal assignment.
\newblock {\em SIGMOD Rec.}, 51(1):51–58.

\bibitem[Addario-Berry and Reed, 2008]{Addario_2008_Ballot}
Addario-Berry, L. and Reed, B.~A. (2008).
\newblock {\em Ballot Theorems, Old and New}, pages 9--35.
\newblock Springer Berlin Heidelberg, Berlin, Heidelberg.

\bibitem[Af\`{e}che et~al., 2022]{Af_queue_2022}
Af\`{e}che, P., Caldentey, R., and Gupta, V. (2022).
\newblock On the optimal design of a bipartite matching queueing system.
\newblock {\em Operations Research}, 70(1):363--401.

\bibitem[Aloqaily et~al., 2022]{Aloqaily_2022_UAV}
Aloqaily, M., Bouachir, O., {Al Ridhawi}, I., and Tzes, A. (2022).
\newblock An adaptive uav positioning model for sustainable smart transportation.
\newblock {\em Sustainable Cities and Society}, 78:103617.

\bibitem[Boeing, 2020]{Boeing_2020_network}
Boeing, G. (2020).
\newblock A multi-scale analysis of 27,000 urban street networks: Every us city, town, urbanized area, and zillow neighborhood.
\newblock {\em Environment and Planning B: Urban Analytics and City Science}, 47(4):590--608.

\bibitem[Boniolo et~al., 2014]{Boniolo_2014}
Boniolo, E., Caracciolo, S., and Sportiello, A. (2014).
\newblock Correlation function for the grid-poisson euclidean matching on a line and on a circle.
\newblock {\em Journal of Statistical Mechanics: Theory and Experiment}, 2014(11):P11023.

\bibitem[Caracciolo et~al., 2019]{Caracciolo_2019}
Caracciolo, S., Di~Gioacchino, A., Malatesta, E.~M., and Molinari, L.~G. (2019).
\newblock Selberg integrals in 1d random euclidean optimization problems.
\newblock {\em Journal of Statistical Mechanics: Theory and Experiment}, 2019(6):063401.

\bibitem[Caracciolo et~al., 2017]{Caracciolo_2017}
Caracciolo, S., D’Achille, M., and Sicuro, G. (2017).
\newblock Random euclidean matching problems in one dimension.
\newblock {\em Physical Review E}, 96(4).

\bibitem[Caracciolo et~al., 2014]{Caracciolo_2014_scaling}
Caracciolo, S., Lucibello, C., Parisi, G., and Sicuro, G. (2014).
\newblock Scaling hypothesis for the euclidean bipartite matching problem.
\newblock {\em Physical Review E}, 90(1).

\bibitem[Caracciolo and Sicuro, 2014]{Caracciolo_2014_matching}
Caracciolo, S. and Sicuro, G. (2014).
\newblock One-dimensional euclidean matching problem: Exact solutions, correlation functions, and universality.
\newblock {\em Phys. Rev. E}, 90:042112.

\bibitem[Caracciolo and Sicuro, 2015]{Caracciolo_2015_scaling}
Caracciolo, S. and Sicuro, G. (2015).
\newblock Scaling hypothesis for the euclidean bipartite matching problem. ii. correlation functions.
\newblock {\em Phys. Rev. E}, 91:062125.

\bibitem[Crouse, 2016]{Scipy}
Crouse, D.~F. (2016).
\newblock On implementing 2d rectangular assignment algorithms.
\newblock {\em IEEE Transactions on Aerospace and Electronic Systems}, 52(4):1679--1696.

\bibitem[Daganzo, 1978]{daganzo_1978}
Daganzo, C.~F. (1978).
\newblock An approximate analytic model of many-to-many demand responsive transportation systems.
\newblock {\em Transportation Research}, 12(5):325--333.

\bibitem[Daganzo and Smilowitz, 2004]{daganzo_tlp_2004}
Daganzo, C.~F. and Smilowitz, K.~R. (2004).
\newblock Bounds and approximations for the transportation problem of linear programming and other scalable network problems.
\newblock {\em Transportation Science}, 38(3):343--356.

\bibitem[Ding et~al., 2021]{Ding_FluidModel_2021}
Ding, Y., McCormick, S.~T., and Nagarajan, M. (2021).
\newblock A fluid model for one-sided bipartite matching queues with match-dependent rewards.
\newblock {\em Operations Research}, 69(4):1256--1281.

\bibitem[Duan and Lu, 2014]{Duan_2014_citynetwork}
Duan, Y. and Lu, F. (2014).
\newblock Robustness of city road networks at different granularities.
\newblock {\em Physica A: Statistical Mechanics and its Applications}, 411:21--34.

\bibitem[Dutta and Dasgupta, 2017]{Dutta_robotPath_2017}
Dutta, A. and Dasgupta, P. (2017).
\newblock Bipartite graph matching-based coordination mechanism for multi-robot path planning under communication constraints.
\newblock In {\em 2017 IEEE International Conference on Robotics and Automation (ICRA)}, pages 857--862.

\bibitem[Ellis, 2011]{ellis2011expansion}
Ellis, D. (2011).
\newblock The expansion of random regular graphs.
\newblock {\em Lecture Notes, Lent}, 34.

\bibitem[Ezaki et~al., 2024]{Ezaki_2024_elevator}
Ezaki, T., Fujitsuka, K., Imura, N., and Nishinari, K. (2024).
\newblock Drone-based vertical delivery system for high-rise buildings: Multiple drones vs. a single elevator.
\newblock {\em Communications in Transportation Research}, 4:100130.

\bibitem[Fedtke and Boysen, 2017]{Fedtke_2017_rail}
Fedtke, S. and Boysen, N. (2017).
\newblock A comparison of different container sorting systems in modern rail-rail transshipment yards.
\newblock {\em Transportation Research Part C: Emerging Technologies}, 82:63--87.

\bibitem[Georgiev and Liò, 2020]{georgiev_neural_2020}
Georgiev, D. and Liò, P. (2020).
\newblock Neural bipartite matching.

\bibitem[Ghassemi and Chowdhury, 2018]{Ghassemi_robotMultiTask_2018}
Ghassemi, P. and Chowdhury, S. (2018).
\newblock {Decentralized Task Allocation in Multi-Robot Systems via Bipartite Graph Matching Augmented With Fuzzy Clustering}.
\newblock volume Volume 2A: 44th Design Automation Conference of {\em International Design Engineering Technical Conferences and Computers and Information in Engineering Conference}, page V02AT03A014.

\bibitem[Goldberg, 1997]{GOLDBERG19971}
Goldberg, A.~V. (1997).
\newblock An efficient implementation of a scaling minimum-cost flow algorithm.
\newblock {\em Journal of Algorithms}, 22(1):1--29.

\bibitem[{Gurobi Optimization, LLC}, 2024]{gurobi}
{Gurobi Optimization, LLC} (2024).
\newblock {Gurobi Optimizer Reference Manual}.

\bibitem[Harel, 1993]{Harel_randomWalkArea_1993}
Harel, A. (1993).
\newblock Random walk and the area below its path.
\newblock {\em Mathematics of Operations Research}, 18(3):566--577.

\bibitem[Hernández et~al., 2021]{Hernandez_2021_UAV}
Hernández, D., Cecília, J.~M., Calafate, C.~T., Cano, J.-C., and Manzoni, P. (2021).
\newblock The kuhn-munkres algorithm for efficient vertical takeoff of uav swarms.
\newblock In {\em 2021 IEEE 93rd Vehicular Technology Conference (VTC2021-Spring)}, pages 1--5.

\bibitem[Jin et~al., 2022]{Jin_schedule_2022}
Jin, C., Xu, J., Han, Y., Hu, J., Chen, Y., and Huang, J. (2022).
\newblock Efficient delay-aware task scheduling for iot devices in mobile cloud computing.
\newblock {\em Mobile Information Systems}, 2022(1):1849877.

\bibitem[Jonker and Volgenant, 1987]{Jonker_alg_1987}
Jonker, R. and Volgenant, A. (1987).
\newblock A shortest augmenting path algorithm for dense and sparse linear assignment problems.
\newblock {\em Computing}, 38(4):325--340.

\bibitem[Kuhn, 1955]{Kuhn_Hungarian_1955}
Kuhn, H.~W. (1955).
\newblock The hungarian method for the assignment problem.
\newblock {\em Naval Research Logistics Quarterly}, 2:83--97.

\bibitem[Lanczos, 1964]{lanczos1964gamma}
Lanczos, C. (1964).
\newblock A precision approximation of the gamma function.
\newblock {\em Journal of the Society for Industrial and Applied Mathematics: Series B, Numerical Analysis}, 1:86--96.
\newblock Accessed 5 Sept. 2025.

\bibitem[MacWilliams and Sloane, 1977]{MacWilliams_1977}
MacWilliams, F.~J. and Sloane, N. J.~A. (1977).
\newblock The theory of error-correcting codes.

\bibitem[Mezard and Parisi, 1985]{Mezard_replica_1985}
Mezard, M. and Parisi, G. (1985).
\newblock Replicas and optimization.
\newblock {\em http://dx.doi.org/10.1051/jphyslet:019850046017077100}, 46.

\bibitem[Mézard and Parisi, 1988]{mezard_euclidean_1988}
Mézard, M. and Parisi, G. (1988).
\newblock The {Euclidean} matching problem.
\newblock {\em Journal de Physique}, 49(12):2019--2025.

\bibitem[Ouyang and Yang, 2023]{ouyang_2023_swap}
Ouyang, Y. and Yang, H. (2023).
\newblock Measurement and mitigation of the “wild goose chase” phenomenon in taxi services.
\newblock {\em Transportation Research Part B: Methodological}, 167:217--234.

\bibitem[Panigrahy et~al., 2020]{Nitish_resouceAllocation_2020}
Panigrahy, N.~K., Basu, P., Nain, P., Towsley, D., Swami, A., Chan, K.~S., and Leung, K.~K. (2020).
\newblock Resource allocation in one-dimensional distributed service networks with applications.
\newblock {\em Performance Evaluation}, 142:102110.

\bibitem[Seth et~al., 2023]{Seth_2023_highrise}
Seth, A., James, A., Kuantama, E., Mukhopadhyay, S., and Han, R. (2023).
\newblock Drone high-rise aerial delivery with vertical grid screening.
\newblock {\em Drones}, 7(5).

\bibitem[Shen and Ouyang, 2023]{shen_swap_2023}
Shen, S. and Ouyang, Y. (2023).
\newblock Dynamic and pareto-improving swapping of vehicles to enhance integrated and modular mobility services.
\newblock {\em Transportation Research Part C: Emerging Technologies}, 157:104366.

\bibitem[Shen et~al., 2024]{shen_zhai_ouyang_2024}
Shen, S., Zhai, Y., and Ouyang, Y. (2024).
\newblock Expected bipartite matching distance in a $d$-dimensional $l^p$ space: Approximate closed-form formulas and applications to mobility services.
\newblock \textit{https://arxiv.org/abs/2406.12174}.

\bibitem[Stiglic et~al., 2015]{Stiglic_2015_meeting}
Stiglic, M., Agatz, N., Savelsbergh, M., and Gradisar, M. (2015).
\newblock The benefits of meeting points in ride-sharing systems.
\newblock {\em Transportation Research Part B: Methodological}, 82:36--53.

\bibitem[Tafreshian and Masoud, 2020]{Tafreshian_ridesharing_2020}
Tafreshian, A. and Masoud, N. (2020).
\newblock Trip-based graph partitioning in dynamic ridesharing.
\newblock {\em Transportation Research Part C: Emerging Technologies}, 114:532--553.

\bibitem[Wang et~al., 2020]{Wang_2020_rideshairngNetwork}
Wang, M., Chen, Z., Mu, L., and Zhang, X. (2020).
\newblock Road network structure and ride-sharing accessibility: A network science perspective.
\newblock {\em Computers, Environment and Urban Systems}, 80:101430.

\bibitem[Wang et~al., 2016]{wang_2016_pricing}
Wang, X., He, F., Yang, H., and {Oliver Gao}, H. (2016).
\newblock Pricing strategies for a taxi-hailing platform.
\newblock {\em Transportation Research Part E: Logistics and Transportation Review}, 93:212--231.

\bibitem[Wang et~al., 2004]{Wang_protein_2004}
Wang, Y., Makedon, F., Ford, J., and Huang, H. (2004).
\newblock A bipartite graph matching framework for finding correspondences between structural elements in two proteins.
\newblock In {\em The 26th Annual International Conference of the IEEE Engineering in Medicine and Biology Society}, volume~2, pages 2972--2975.

\bibitem[Werman et~al., 1986]{werman_1986}
Werman, M., Peleg, S., Melter, R., and Kong, T. (1986).
\newblock Bipartite graph matching for points on a line or a circle.
\newblock {\em Journal of Algorithms}, 7(2):277--284.

\bibitem[Yang et~al., 2010]{yang_CobbDoglous_2010}
Yang, H., Leung, C., Wong, S., and Bell, M. (2010).
\newblock Equilibria of bilateral taxi–customer searching and meeting on networks.
\newblock {\em Transportation Research Part B: Methodological}, 44:1067--1083.

\bibitem[Zhang et~al., 2014]{Zhang_bicliques_2014}
Zhang, Y., Phillips, C.~A., Rogers, G.~L., Baker, E.~J., Chesler, E.~J., and Langston, M.~A. (2014).
\newblock On finding bicliques in bipartite graphs: a novel algorithm and its application to the integration of diverse biological data types.
\newblock {\em BMC Bioinformatics}, 15(1):110.

\bibitem[Zhou et~al., 2022]{zhou_2022_competition}
Zhou, Y., Yang, H., and Ke, J. (2022).
\newblock Price of competition and fragmentation in ride-sourcing markets.
\newblock {\em Transportation Research Part C: Emerging Technologies}, 143:103851.

\end{thebibliography}

\newpage
\appendix

\renewcommand{\thesection}{Appendix \Alph{section}}

\section{Proof for Equation \eqref{periodic_cond_3}}
\label{apx: optimal}
Let random variable $M_{n,K}$ denote the number of times a random walk with $2n$ steps reaches a value of $K$. 
The random walk could reach $K$ at the $i$-th step if these $i$ steps include $\frac{i+K}{2}$ upward steps and $\frac{i-K}{2}$ downward steps.
Since this random walk starts from $0$, we must have $i \ge K$, and $i-K$ must be an even number; i.e., $i\equiv K\Mod{2}$. 
As such, the probability for the random walk to reach $K$ exactly at the $i$-th step (which is a Bernoulli random variable)
is
\begin{align*}
\frac{\binom{n}{(i+K)/2} \binom{n}{(i-K)/2}}{\binom{2n}{i}}, \quad \forall i \geq K, i\equiv K\Mod{2}.
\end{align*}
Then, $\mathbb{E}[M_{n, k}]$ can be obtained by summing the probabilities for all the 
Bernoulli random variables, as follows: 
\begin{align*}
\mathbb{E}[M_{n,K}] = \sum_{\{\forall i \geq K, i\equiv K\Mod{2}\}} 
\frac{\binom{n}{(i+K)/2} \binom{n}{(i-K)/2}}{\binom{2n}{i}}. 
\end{align*}
Given any $i$, the corresponding probability represents the probability mass function of a hyper-geometric distribution, which is known to be unimodal. It is easy to show that its maximum is taken at $K = 0$. Thus, we have:
\begin{align*}
\mathbb{E}[M_{n,K}] \leq 
\sum_{\{\forall i\equiv 0 \Mod{2}\}}\frac{\binom{n}{i/2} \binom{n}{i/2}}{\binom{2n}{i}}.
\end{align*}
A binomial coefficient is known to have the following bounds \citep{MacWilliams_1977}:
\begin{align*}
\sqrt{\frac{2n}{8 i(2n-i)}}2^{2nH(i/(2n))}\leq \binom{2n}{i} \leq \sqrt{\frac{2n}{2\pi i(2n-i)}}2^{2nH(i/(2n))}, \ \forall 1 \leq i \leq 2n-1,
\end{align*}
where $H(\cdot)$ is the binary entropy function.
The upper bound of $\mathbb{E}[M_{n,K}]$ can be further relaxed to the following: 
\begin{align*}
\mathbb{E}[M_{n,K}]
&\leq 
2 + \sum_{\{i=2, \dots, 2n-2\}}
\frac
{\left(\sqrt{\frac{n}{2\pi i/2(n-i/2)}}2^{nH(i/(2n))}\right)^2}
{\sqrt{\frac{2n}{8 i(2n-i)}}2^{2nH(i/(2n))}}
= 2+ \frac{2\sqrt{2}}{\pi}\sum_{\{i=2, \dots, 2n-2\}}
\sqrt{\frac{2n}{i(2n-i)}}.\\
&\leq 2+ \frac{4}{\pi}\sum_{\{i=2, \dots 2n-2\}}
\sqrt{\frac{1}{\min \{i, 2n-i\}}}
\leq 2+ \frac{8}{\pi}\sum_{\{i=2, \dots 2\lceil \frac{n}{2}\rceil\}}
\sqrt{\frac{1}{i}} 
\leq 2+ \frac{8}{\pi}\cdot 2\sqrt{n+1}. 
\end{align*}
%
Given $l = \frac{1}{2n}$, we have
$$
\mathbb{E}\left[\sum_{\{i\in I:S(x;I) - K=0\}}l\right] = 
\mathbb{E}[l\cdot M_{n, K}] 
\le \frac{1}{2n} (2+ \frac{8}{\pi}\cdot 2\sqrt{n+1}),
$$
which clearly converges to $0$ when $n$ goes to infinity. 
This 
completes the proof.
\newpage
\section{Solution Approaches for 1D RBMP Instances}
\label{apx: DP}
While any instances of 1D RBMPs can be solved by standard linear programming techniques in polynomial time, this section presents alternative methods that may be computationally more efficient for large-scale cases.

First, for balanced cases with open boundaries, each problem instance can be solved easily using the approach described in Section \ref{subsec: approximation}, where the total optimal matching distance equals the area size under the cumulative net supply curve. This approach has a worst-case time complexity of $\mathcal{O}(n)$.
Second, for balanced case with periodic boundaries, each problem instance can be solved via Equation \eqref{periodic_min} by simply enumerating all points in $I$ for shifting the cumulative net supply curve. 
This approach has a worst-case time complexity of $\mathcal{O}(n^2)$.

Finally, for unbalanced case with open boundaries, 
the total matching distance equals the area size under the cumulative net supply curve after removing the $n-m$ excessive supply points. We propose a dynamic programming (DP)-based approach, which has a worst-case time complexity of $\mathcal{O}(n^3)$
, as follows. 
We define $n-m$ stages, indexed by $k \in \{1, 2, \dots, n-m\}$, each for one point removal.
The state of the system at stage $k$ is represented by the location of a point in $U\cup V$ along the x-axis, denoted as $s_{k}\in [0,1]$, such that all points located in $[0,s_{k}]$ have been ``inspected for removal." 
The action at stage $k$ is to find the next supply point $v_k \in V$, satisfying $x_{v_k} > s_k$ and $S(x_{v_k};I) = k$ from Proposition \ref{prop: S_eq_k}, to remove. 
The system state then transitions to the location of the newly removed point $v_k$ at the next stage;
i.e., $s_{k+1} := x_{v_k}$.
The cost associated with this single removal, denoted $c_k(v_k)$, is computed as the area under the cumulative net supply curve between interval $(s_k, x_{v_{k}})$, considering all selected supply points from stages $1$ to $k$ are removed. 
The initial state in stage 1 is at the location of the leftmost point in  $U\cup V$.
At the final stage when $k = n-m$, the terminating cost must include an additional area from interval $(x_{v_{n-m}}, 1)$. 
Let $\mathbf{V}_k(s_k)$ denote the value function of state $s_k$ at stage $k$, which represents the tail area under the post-removal curve over interval $(s_k, 1)$.
The optimal point removal decisions can be obtained by solving the following Bellman equation using backward induction. 
\begin{align*}
&\mathbf{V}_{k}(s_k) = \min\limits_{\substack{ \{v_{k}\in V :\ x_{v_{k}} > s_{k}, \\ S(x_{v_{k}};I) = k\}}} 
      \left\{ 
      c_k(v_k) + \mathbf{V}_{k+1}(s_{k+1}) 
      \right\}
      \quad \forall k \in \{1,\dots n-m-1\},\\
&\mathbf{V}_{n-m}(s_{n-m}) = \min\limits_{\substack{ \{v_{n-m}\in V :\ x_{v_{n-m}} > s_{n-m}, \\ S(x_{v_{n-m}};I) = n-m\}}}  \{c_{n-m}(s_{n-m})\}.
\end{align*}

{\color{Black}
Next, we compare the performance of the proposed DP algorithm with that of a state-of-art modified Jonker-Volgenant algorithm \citep{Scipy} on a lattice with an open boundary condition. 
\footnote{\textcolor{Black}{We have also tested with Gurobi \citep{gurobi} and a cost scaling method \citep{GOLDBERG19971}, both of which are outperformed by the modified Jonker-Volgenant algorithm}}
We set $m \in \{100, 500,$ $1000, 5000, 10000\}$ and consider three representative cases for $n$: nearly-balanced case, $n = 1.5m$; unbalanced case, $n = 2m$; and highly-unbalanced case, $n = 6m$. For each $(m,n)$ combination, we randomly generate 100 instances and 
compute the average solution time across these 100 instances. 

The results are shown in Figure \ref{fig:dp}. 
Average running times of the DP and modified Jonker-Volgenant algorithm are represented by the red solid line with circle markers and blue dash line with cross markers, respectively. It is observed that the proposed DP method has a promising performance: it notably outperforms the modified Jonker-Volgenant algorithm especially when the $n/m$ ratio is small (e.g., $\le 2$) and the problem size is large 
(e.g., $m\geq 5000$). When the problem is highly unbalanced (e.g., $n/m=6$), DP has a similar average running time as the modified Jonker-Volgenant algorithm. This is also unsurprising, as the solutions for highly unbalanced problems are generally easier to obtain.
}

\begin{figure}
    \centering
    \includegraphics[width=1\linewidth]{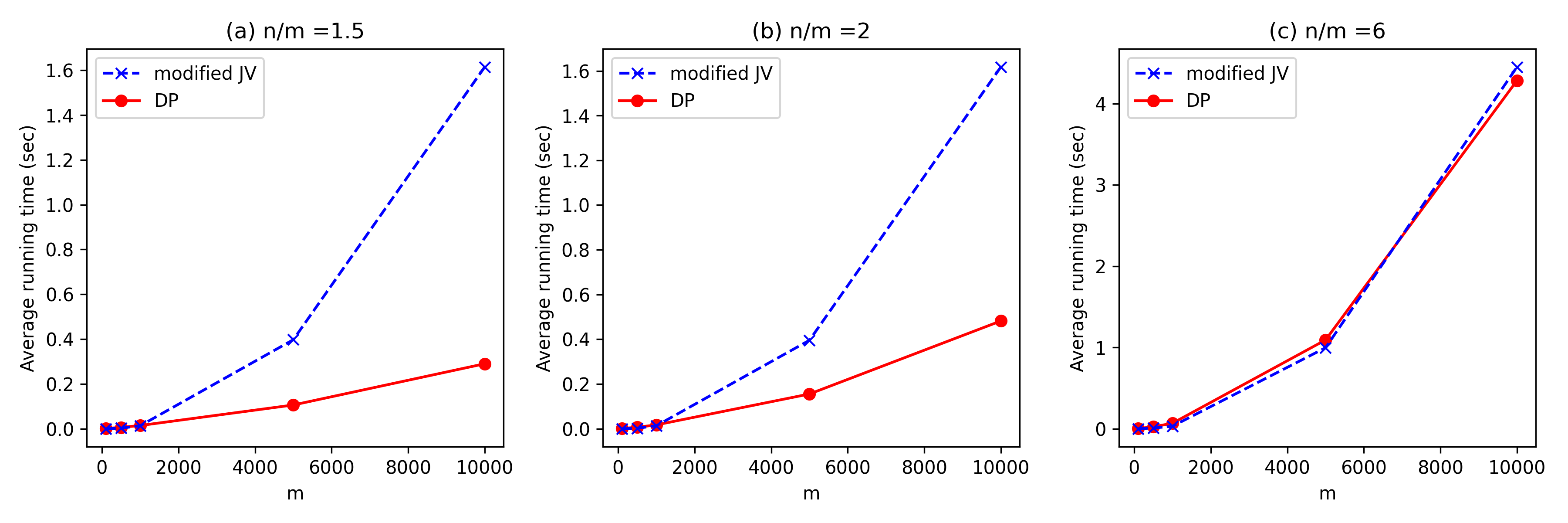}
    \caption{Running time comparison of dynamic programming method.}
    \label{fig:dp}
\end{figure}

\pagebreak

\section{List of notation}
\nopagebreak
\renewcommand{\arraystretch}{1}
\begin{table}[ht]
\centering
\begin{tabular}{p{2cm}p{14cm}}
\toprule
\toprule
\textbf{Notation} & \textbf{Description} \\
\midrule
$m, n$ & Cardinalities of the two point sets in an one-dimensional RBMP \\
$U$, $V$ & Two sets of points for each one-dimensional RBMP realization\\
$I$ & Set containing all points in $U$ and $V$ \\
$E$ & Set of edges connecting $u\in U$ and $v\in V$ \\
$y_{uv}$ & 1 if $u\in U$ is matched to $v\in V$, 0 otherwise \\
$x_i$ & x-coordinate of a point $i\in I$\\
$z_i$ & Supply value of a point $i\in I$ \\
$X_{m,n}$ & Average optimal matching distance per point in a lattice problem\\
$l$ & Step size between any two consecutive points in $I$ in a lattice problem\\
$l_i$ & Step size from a point $i\in I$ to its next point in a problem with uniform point distribution\\
$S(x;I')$ & Net supply curve for any coordinate $x$ and subset of points $I'\subseteq I$ \\
$A(x;I')$ & Total absolute area between curve $S(x;I')$ and x-axis from 0 to $x$ for lattice case\\
$A^\text{u}(x;I')$ & Total absolute area between curve $S(x;I')$ and x-axis from 0 to $x$ for uniform point distribution case\\
$B(n)$ & Expected total absolute area between the path of a random walk with $2n$ unit-length steps and x-axis \\
$V'$ & An arbitrary set of unmatched/removed points in $V$ \\
$V^*$ & Optimal set of removed points \\
$\hat{V}$ & Set of removed points from the proposed point removal process\\
$v'_k, v^*_k, \hat{v}_k$ & $k$-th removed point along the x-axis in $V'$, $V^*$, $\hat{V}$ \\
$\textbf{l}$ & Vector of step sizes with uniform point distribution\\
$\textbf{s}(V')$ & Vector of the absolute values of post-removal cumulative net supply\\
$m_k$ & Number of demand points in the $k$-th segment of the post-removal curve $S(x; I\setminus V^*)$\\
$\hat{m}_k$ & Number of demand points in the $k$-th segment of the post-removal curve $S(x; I\setminus \hat{V})$\\
$\hat{m}_{k, 0}$ & Number of demand points with zero net supplies in the $k$-th segment of the post-removal curve $S(x; I\setminus \hat{V})$ \\
$\hat{Z}_{k,a}$ & Total absolute area enclosed by $S(x; I\setminus \hat{V})$ within $(x_{\hat{v}_k}, 1]$, which contains exactly $a$ demand points \\
$Z_{m,n}$ & Total absolute area enclosed by $S(x; I\setminus V^*)$ from 0 to 1\\
$X^\circ_{n,n}$ & Average optimal matching distance per point in a periodic-boundary problem\\
$i^*$ & optimal point in $I$ to minimize the area under the shifted cumulative net supply curve in a periodic-boundary problem\\
\bottomrule
\bottomrule
\end{tabular}
\end{table}
\pagebreak
\begin{table}[ht]
\centering
\begin{tabular}{p{2cm}p{14cm}}
\toprule
\toprule
\textbf{Notation} & \textbf{Description}\\
\midrule
$M_{n,K}$ & Number of times a random walk with $2n$ steps reaches a value of $K$\\
$X^\text{u}_{m,n}$ & Average optimal matching distance per point in a problem with uniform point distribution\\
$L$ & Length of a line (edge) \\
$X_{\text{E}}$ & Average optimal matching distance per point on an arbitrary-length line \\
$\mu, \lambda$ & Densities of the two point sets on an arbitrary-length line \\
$\text{G} = (\mathcal{V}, \mathcal{E})$& Graph with node set $\mathcal{V}$ and edge set $\mathcal{E}$\\
$D$ & Degree of node in $\mathcal{V}$ \\
$U_e, V_e$ & Realized point sets on an edge $e\in \mathcal{E}$  \\
$m_e, n_e$ & Cardinality of point sets $U_e$ and $V_e$ \\
$X_\text{G}$ & Average optimal matching distance per point in a graph\\
$X^\text{l}_\text{G}, X^\text{g}_\text{G}$ & Average optimal local and global matching distances in a graph\\
$\alpha$ & Probability for global matching\\
$\phi(\cdot), \Phi(\cdot)$ & Probability density function and cumulative distribution function of standard normal distribution\\
$U_e^+, V_e^+$ & Remaining point sets on an edge $e\in \mathcal{E}$ after local matching process\\
$N_k$ & $k$-th layer of edges for a point $u\in U_e^+$ in breadth-first search\\
$e^*$ & Edge containing the matched point $v$ for a point $u\in U_e^+$ \\
$o_0$ & Nearer end of $e$ to $u\in U_e^+$ \\
$o_k$ & Nearer end of $e^*$ to $o_0$ \\
$d_1$ & Distance from $u\in U_e^+$ to $o_0$ \\
$d_2$ & Distance from $o_0$ to $o_k$ \\
$d_3$ & Distance from $o_k$ to the matched point $v$ of $u\in U_e^+$\\
$s_k$ & System state at stage $k$ \\
$c_k(\cdot)$ & Cost function at stage $k$ \\
$\mathbf{V}_k(\cdot)$ & Value function of state $k$\\
$N$ & Sample size of numerical experiments\\
\bottomrule
\bottomrule
\end{tabular}
\end{table}
\pagebreak

\color{Black}
\end{document}